\documentclass[11pt]{article}

\usepackage{jcd}

\usepackage{mathtools}
\usepackage{color}
\usepackage{comment}
\usepackage{xspace}
\usepackage{tikz}

\usepackage{overpic}

\usepackage[title]{appendix}

\definecolor{darkblue}{rgb}{0,0,.75}
\usepackage{graphicx}
\usepackage[colorlinks=true,linkcolor=darkblue,%
citecolor=darkblue,urlcolor=darkblue]{hyperref}
\usepackage[numbers]{natbib}

\usepackage{enumitem}
\usepackage[top=2.54cm,right=3cm,left=3cm,bottom=2.54cm]{geometry}

\newcommand{\statrv}{S}
\newcommand{\statval}{s}
\newcommand{\statdomain}{\mc{S}}
\newcommand{\stepsize}{\alpha}
\newcommand{\opt}{^\star}
\newcommand{\subopt}{_\star}

\newcommand{\growfunc}{\mathsf{G}_{\rm big}}

\providecommand{\steppow}{\beta}

\newcommand{\lipgrad}{L}

\newcommand{\weakconvexfunc}{\rho}

\newcommand{\aProx}{\textsc{aProx}\xspace}

\newcommand{\stationary}{_{\rm stat}}
\newcommand{\prox}{{\rm prox}}

\begin{document}
	
\begin{center}
	{\Large The importance of better models in stochastic optimization}
	
	\vspace{.3cm}
	
	\begin{tabular}{cc}
		Hilal Asi & John C.\ Duchi \\
		\texttt{asi@stanford.edu} &
		\texttt{jduchi@stanford.edu} \\
		\multicolumn{2}{c}{
			Stanford University
		} \\
	\end{tabular}
\end{center}

%\maketitle

\begin{abstract}
  Standard stochastic optimization methods are brittle, sensitive to
  stepsize choices and other algorithmic parameters, and they exhibit
  instability outside of well-behaved families of objectives. To address
  these challenges, we investigate models for stochastic optimization and
  learning problems that exhibit better robustness to problem families and
  algorithmic parameters. With appropriately accurate models---which we call
  the \aProx family~\cite{AsiDu18}---stochastic methods can be made stable,
  provably convergent and asymptotically optimal; even modeling
  that the objective is nonnegative is sufficient for this
  stability. We extend these results beyond convexity to weakly convex
  objectives, which include compositions of convex losses with smooth
  functions common in modern machine learning applications. We highlight the
  importance of robustness and accurate modeling with a careful experimental
  evaluation of convergence time and algorithm sensitivity.
\end{abstract}

%% \begin{abstract}
%%   We develop model-based method for solving stochastic (possibly) non-convex
%%   optimization problems, namely \aProx, which includes stochastic
%%   subgradient, proximal point, and bundle methods.  We propose and
%%   investigate new models which enjoy better robustness guarantees than the
%%   standard stochastic subgradient, henceforth requiring less parameter
%%   tuning efforts, even though the model-based methods typically add little
%%   to no computational overhead over stochastic subgradient methods.  We
%%   provide convergence guarantees for \aProx for convex and weakly convex
%%   functions, emphasizing the importance of more accurate models. We also
%%   prove that the proposed models are adaptive and enjoy linear convergence
%%   on a class of stochastic optimization problems, which we term easy
%%   problems.  To confirm our theory, we present several experimental results
%%   which demonstrate the advantages of more accurate modeling over standard
%%   subgradient methods across many non-convex optimization problems.
%% \end{abstract}

%\tableofcontents

% -*- Mode: latex -*- %

\section{Introduction}
\label{sec:intro}

A major challenge in stochastic optimization---the algorithmic workhorse for
much of modern statistical and machine learning applications---is in setting
algorithm parameters (or hyper-parameter tuning).
The challenge arises because most algorithms are sensitive to
their parameters, and different applications require different tuning. This
sensitivity causes multiple issues. It results in thousands to
millions of wasted engineer and computational hours. It also leads to a
lack of clarity in research and development of algorithms---in claiming
that one algorithm is better than another, it is unclear if this is due to
judicious choice of dataset, judicious parameter settings,
or if indeed the algorithm
does exhibit new desirable behavior. Consequently, in this paper we have two
main thrusts: first, by using better models than
naive first-order models common in stochastic gradient methods,
we develop families of stochastic optimization
algorithms that are provably more robust to input parameter choices, with
several corresponding optimality properties. Second, we argue for a
different type of experimental evidence in evaluating stochastic
optimization methods, where one jointly evaluates convergence speed and
sensitivity of the methods.

The wasted computational and engineering energy is especially pronounced in
deep learning, where engineers use models with millions of parameters,
requiring a days to weeks to train single models.  To get a sense of this
energy use, we consider a few recent papers that we view as exemplars of
this broader trend: in searching for optimal neural network architectures
and hyperparameters, the papers~\cite{ZophVaShLe18, RealAgHuLe18, ZophLe16}
used approximately 2000, 3150, 22000 GPU-days of computation,
respectively. The paper~\cite{CollinsSoSu16} uses approximately 750000
CPU days in its parameter search.  To put this in
perspective, assuming standard CPU energy use of between 60-100 Watts,
the energy (ignoring network interconnect, monitors, etc.)
for the paper~\cite{CollinsSoSu16} is roughly between $4$ and $6
\cdot 10^{12}$ Joules. At $10^9$ Joules per tank of gas, this is
sufficient to drive 4000 Toyota Camrys the $380$ miles between San
Francisco and Los Angeles.

%% different perspective, the amount of Joules spent by
%% Google researchers~\cite{CollinsSoSu16} is about (as standard CPU uses about 65-85 Watts, 
%% or 234000 Joules per hour) $4 \cdot 10^{12}$ Joules; an enormous
%% amount of energy sufficient to supply more than $1000$ cars with
%% gas to travel $380$ miles from San Fransisco to Los Angeles.

To address these challenges,
we develop stochastic optimization procedures that exhibit similar convergence
to classical approaches---when the classical approaches are
provided good tuning parameters---but they enjoy
better robustness and achieve this performance over a range of
parameters. We additionally argue for
evaluation of optimization algorithms
based not only on convergence time, but also on robustness
to input choices. Briefly, a fast algorithm that converges for
a small range of stepsizes is too brittle; we argue
instead for (potentially slightly slower) algorithms that
converge for broad ranges of stepsizes and other parameters.
Our theory and experiments demonstrate the effectiveness of our methods
for many applications, including phase retrieval,
matrix completion, and deep learning.

\subsection{Problem setting and approach}

We begin by making our setting concrete. We study the stochastic
optimization problem
\begin{equation}
  \label{eqn:objective}
  \begin{split}
    \minimize ~ & F(x) \defeq \E_P[f(x;\statrv)]
    = \int_\statdomain f(x; \statval) dP(\statval) \\
    \subjectto ~ & x \in \mc{X}.
  \end{split}
\end{equation}
In problem~\eqref{eqn:objective}, the set $\statdomain$ is a sample space,
$\mc{X} \subset \R^n$ is a closed convex, and $f(x; \statval)$ is the
instantaneous loss parameter $x$ suffers on sample $\statval$.  In this
paper, we move beyond convex optimization by considering
$\weakconvexfunc(\statval)$-weakly convex functions $f$,
meaning~\cite[cf.][]{RockafellarWe98, Drusvyatskiy18} that
\begin{equation*}
  x \mapsto \left\{f(x; \statval) + \frac{\weakconvexfunc(\statval)}{2}
  \ltwo{x}^2 \right\}
  ~~ \mbox{is convex}.
\end{equation*}
We recover the convex case when $\weakconvexfunc(\statval) \le 0$.
Examples in this framework include linear regression, where
$f(x;(a,b)) = ( \<a,x\> - b)^2$, robust phase
retrieval~\cite{SchechtmanElCoChMiSe15,DuchiRu18a} where $f(x;(a,b)) =
|\<a,x\>^2 - b|$, which is $2\ltwo{a}^2$-weakly convex, or bilinear
prediction, $f(x, y; b) = |\<x, y\> - b|$, which is $1$-weakly convex.

Most optimization methods iterate by making an approximation---a model---of
the objective near the current iterate, then minimizing this model and
re-approximating.  Stochastic (sub)gradient methods~\cite{RobbinsMo51,
  NemirovskiJuLaSh09} instantiate this approach using a linear
approximation; following initial work of our own and
others~\cite{AsiDu18,DuchiRu18c,DavisDr19}, we study the modeling approach
in more depth for stochastic optimization. Thus, the \aProx algorithms we
develop~\cite{AsiDu18} iterate as follows: for $k = 1, 2, \ldots$, we draw a
random $\statrv_k \sim P$, then update the iterate $x_k$ by minimizing a
regularized approximation to $f(\cdot; \statrv_k)$, setting
\begin{equation}
\label{eqn:model-iteration}
x_{k+1} \defeq \argmin_{x \in \mc{X}}
\left\{ f_{x_k}(x ; \statrv_k) + \frac{1}{2 \stepsize_k}
\ltwo{x - x_k}^2 \right\}.
\end{equation}
We call $f_x(\cdot ; \statval)$ the \emph{model} of $f$ at
$x$, where $f_x$ satisfies three
conditions~\cite[cf.][]{AsiDu18,DuchiRu18c,DavisDr19}:
\begin{enumerate}[label=(C.\roman*),leftmargin=*]
\item \label{cond:convex-model}
  [Model convexity]
  The function $y \mapsto f_x(y; \statval)$ is convex
  and subdifferentiable.
\item \label{cond:weak-convex-lower}
  [Weak lower bound]
  The model $f_x$ satisfies
  \begin{equation*}
    f_x(y; \statval) \le f(y; \statval) + \frac{\weakconvexfunc(\statval)}{2}
    \ltwo{y - x}^2
    ~~ \mbox{for~all~} y \in \mc{X}.
  \end{equation*}
\item \label{cond:subgrad-model}
  [Local accuracy]
  We have $f_x(x; \statval) = f(x; \statval)$ and the containment
  \begin{equation*}
    \left.\partial_y f_x(y; \statval)\right|_{y = x}
    \subset \partial_x f(x; \statval).
  \end{equation*}
\end{enumerate}
\noindent
We provide examples in Section~\ref{sec:methods}.

We show that by using just slightly more accurate models than the first
order model used by the stochastic gradient method---sometimes as simple as
recognizing that if the function loss $f$ is non-negative, we should
truncate our approximation at zero---we can achieve substantially better
theoretical guarantees and practical performance.  While the iterates of
gradient methods can (super-exponentially) diverge as soon
as we have mis-specified
stepsizes, our methods guarantee that the iterates never diverge.  Even
more, this stability guarantees convergence of the methods,
and in convex cases, optimal asymptotic normality of the averaged
iterates. Finally, we
evaluate the performance of our methods, reaffirming our
theoretical findings on convergence and
robustness for a range of problems, including matrix
completion, phase retrieval, and classification with neural networks.
We defer all proofs to the appendices.

\paragraph{Notation and basic assumptions}
For a weakly convex function $f$, we let $\partial f(x)$ denote its
Fr\'echet subdifferential at the point $x$, and $f'(x) \in \partial f(x)$
denotes an arbitrary element of the subdifferential.  Throughout, we let
$x\opt$ denote a minimizer of problem~\eqref{eqn:objective} and $\mc{X}\opt
= \argmin_{x \in \mc{X}} F(x)$ denote the optimal set for
problem~\eqref{eqn:objective}. We let $\mc{F}_k \defeq \sigma(\statrv_1,
\ldots, \statrv_k)$ denote the $\sigma$-field generated by the first $k$
random variables $\statrv_i$. Note that $x_k \in \mc{F}_{k-1}$ for all
$k$. Unless stated otherwise, we assume that the function $f(x;\statval)$ is
$\weakconvexfunc(\statval)$-weakly convex for each $s\in
\statdomain$. Finally, the following assumption will implicity hold
throughout the paper.
\begin{assumption}
  \label{assumption:bounded-noise}
  The set $\mc{X}\opt \defeq \argmin_{x \in \mc{X}}
  \{F(x)\}$ is non-empty, and there exists
  $\sigma^2 < \infty$ such that for
  each $x\opt \in \mc{X}\opt$ and selection $f'(x\opt; \statval)
  \in \partial f(x\opt; \statval)$, we have
  $\E[\ltwo{f'(x\opt;\statrv)}^2] \leq \sigma^2$.
\end{assumption}

% -*- Mode: latex -*- %

\section{Methods}
\label{sec:methods}

\begin{comment}
\begin{figure}[t]
  \begin{center}
    \begin{tabular}{cc}
      \hspace{-.5cm}
      \begin{overpic}[width=.5\columnwidth]{%,grid]{%
          plots/lower-models}
        \put(18,65.5){\footnotesize{Linear}}
        \put(18,58.5){\footnotesize{Truncated}}
        \put(49.4,37){\footnotesize{$x_0$}}
        \put(73,55.5){\footnotesize{$x_1$}}
      \end{overpic} &
      \includegraphics[width=.5\columnwidth]{plots/multiline-model}
      \\
      (a) & (b)
    \end{tabular}
    \caption{\label{fig:model-illustrations} (a) Models of the function
      $f(x) = \log(1 + e^{-x})$: a linear
      model~\eqref{eqn:dumb-linear-model} built around the point $x_0$ and
      truncated model~\eqref{eqn:trunc-model} built around the point
      $x_1$. (b) The multi-line (or bundle) model, maximum of linear
      functions, as in the iteration~\eqref{eqn:bundling}. The
      lighter lines represent individual linear approximations, the
      darker line their maximum.}
  \end{center}
\end{figure}
\end{comment}

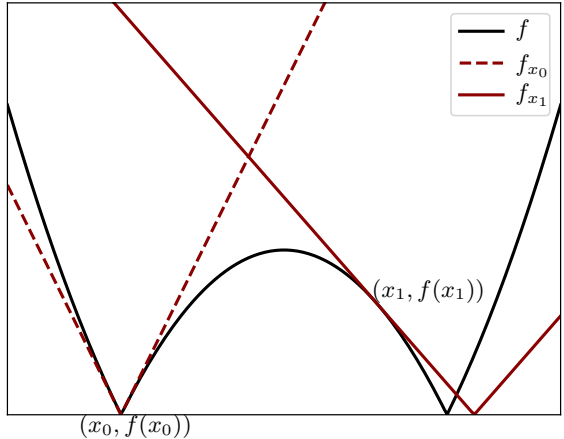
\begin{figure}[t]
  \begin{center}
    \begin{tabular}{cc}
      \hspace{-.5cm}
      \begin{overpic}[width=.5\columnwidth]{%,grid]{%
	  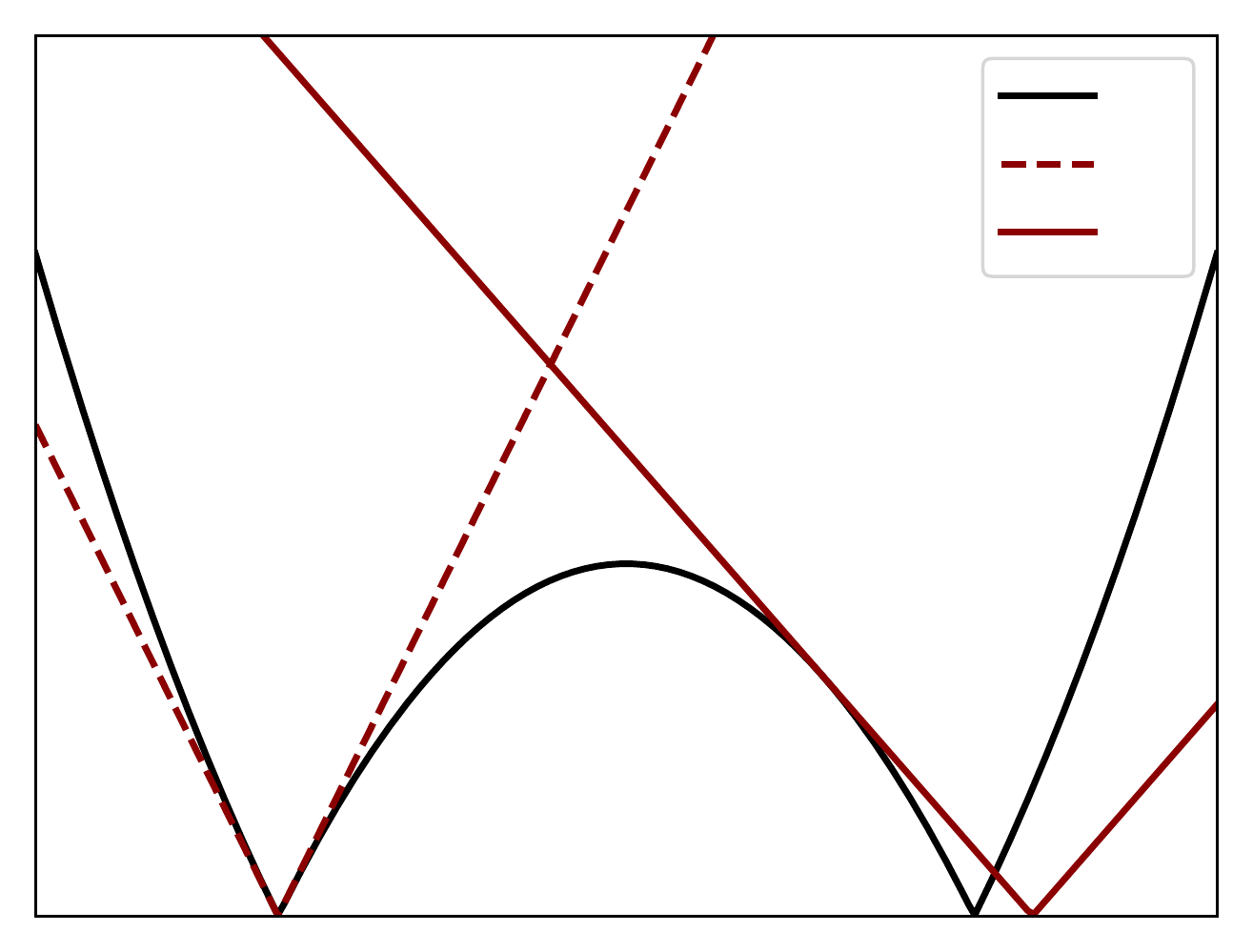}
  	 \put(89.5,67.5){\footnotesize{$f$}}
  	 \put(89.5,61.5){\footnotesize{$f_{x_0}$}}
  	 \put(89.5,56.5){\footnotesize{$f_{x_1}$}}
	%\put(18,65.5){\footnotesize{Linear}}
	%\put(18,58.5){\footnotesize{Truncated}}
	  \put(15,-0.15){\footnotesize{$(x_0,f(x_0))$}}
	  \put(65,23){\footnotesize{$(x_1,f(x_1))$}}
	%\put(63,0){\footnotesize{$x_1$}}
      \end{overpic} 
    \end{tabular}
    \caption{\label{fig:prox-linear-illustrations} Prox-linear model of the
      function $f(x) = | x^2 - 1 |$ at the points $x_0$ and $x_1$.}
  \end{center}
\end{figure}

%We introduce different natural models for \aProx, some of which were introduced in~\cite{AsiDu18} for the convex setting, and additional models which are suited for the non-convex setting. We also define a few conditions on the accuracy of the models, which are sufficient to prove better convergence guarantees for the \aProx family. While the stochastic proximal point method~\eqref{eqn:prox-model} satisfies all the conditions in the paper, in some situations it may be expensive or challenging to implement exactly.  With that in mind, we provide a catalogue of a few models to serve as a reference for the remainder of the paper.

To make our approach a bit more concrete, we identify several models that
fit into our framework.  These have
appeared~\cite{DuchiRu18c,DavisDr19,AsiDu18}, but we believe a
self-contained presentation beneficial.  Each of these models satisfies our
conditions~\ref{cond:convex-model}--\ref{cond:subgrad-model}.  The most
widely used model in stochastic optimization, is the simple first-order
model:

\paragraph{Stochastic subgradient methods:}
The stochastic subgradient method
uses the model
\begin{equation}
  \label{eqn:dumb-linear-model}
  f_{x}(y; \statval) \defeq f(x; \statval) + \<f'(x; \statval), y - x\>.
\end{equation}

\paragraph{Proximal point methods:}
In the convex setting~\cite{Bertsekas11, PatrascuNe17, AsiDu18}, the
stochastic proximal point method uses the model $f_x(y; \statval) \defeq
f(y; \statval)$; in the weakly convex setting, we regularize and use
\begin{equation}
  \label{eqn:prox-model}
  f_x(y; \statval) \defeq f(y; \statval) + \frac{\weakconvexfunc(\statval)}{2}
  \ltwo{y - x}^2.
\end{equation}

We now turn to models that require less knowledge than proximal model~\eqref{eqn:prox-model}, but preserve 
important structural properties in the original function.
\paragraph{Prox-linear model:} 
%The motivation for the prox-linear model stems from many real-life 
%applications (to be investigated later) in which the optimization 
%function can be written as a composition of several other functions.
Let the function $f$ have the composite structure 
$f(x;\statval) = h(c(x;\statval);\statval)$ where $h(\cdot;\statval)$ 
is convex and $c(\cdot;\statval)$ is smooth.
The stochastic prox-linear method applies $h$ 
to a first-order approximation of $c$, using
\begin{equation}
  \label{eqn:prox-linear-model}
  f_x(y; \statval) \defeq
  h(c(x;\statval) + \nabla{c(x;\statval)}^T (y-x); \statval).
\end{equation}
In the non-stochastic setting, these models are
classical~\cite{Fletcher82,FletcherWa80}, while in stochastic setting,
recent work establishes convergence and convergence rates in restrictive
settings~\cite{DuchiRu18c,DavisDr19}.  See
Figure~\ref{fig:prox-linear-illustrations} for illustration of this model.
When $h$ is $L_h$-Lipschitz and $c$ has $L_c$-Lipschitz gradient, then $f$
is $\weakconvexfunc = L_h \cdot L_c$-weakly convex.

%The prox-linear model provides a way to construct a convex approximation 
%of a (possibly) non-convex function with a specific structure.
%Fletcher and Watson~\cite{Fletcher82,FletcherWa80} developed the prox-linear method in the non-stochastic setting, 
%and more recently~\citet{DuchiRu17} investigated the stochastic version of the method. 
%In this setting, we assume the functions $f$ have the composite structure 
%$f(x;\statval) = h(c(x;\statval);\statval)$ where $h(\cdot;\statval)$ is convex 
%and $c(\cdot;\statval)$ is smooth. The prox-linear method applies $h$ 
%to a linear model of $c$, i.e. it uses the following model
%\begin{equation}
%\label{eqn:prox-linear-model}
%f_x(y; \statval) \defeq h(c(x;\statval) + \nabla{c^T(x;\statval)} (y-x); \statval).
%\end{equation}
%We note that the prox-linear model~\eqref{eqn:prox-linear-model} will oftentimes satisfy the lower bound condition~\ref{cond:lower-by-optimal}.
%See Figure~\ref{fig:prox-linear-illustrations} for illustration of this model.

Examples help highlight the applicability of this composite structure:

\begin{example}[Phase Retrieval]
  \label{example:PR}
  In phase retrieval~\cite{SchechtmanElCoChMiSe15}, we wish to recover an
  object $x\opt \in \C^n$ from a diffraction pattern $Ax\opt$, where $A \in
  \C^{m \times n}$, but physical sensor limitations mean we only observe the
  amplitudes $b=|Ax\opt|^2$.  A natural objective is thus
  \begin{equation*}
    \minimize_{x \in \C^n} \frac{1}{m} \sum_{i=1}^{m}
    f(x; (a_i, b_i))
    ~~~ \mbox{where} ~~
    f(x; (a_i, b_i)) \defeq \left|
    |\<a_i, x\>|^2 - b_i\right|.
  \end{equation*}
  This is evidently the composition of
  $h(z)=|z|$ and $c(x;(a_i,b_i))= |\<a_i,x\>|^2 - b_i$. 
  Moreover, a calculation~\cite{DuchiRu18a} shows that
  $f(\cdot;(a_i,b_i))$ is $2\ltwo{a_i}^2$-weakly convex.
\end{example}

%{KeshavanMoOh10,CandesRe08}
\begin{example}[Matrix Completion]
  \label{example:matComp}
  In the matrix completion problem~\cite{CandesRe08}, which arises (for
  example) in the design of recommendation systems, we have a matrix $M\in
  \R^{m \times n}$ with decomposition $M = X\subopt Y\subopt^T$ for
  $X\subopt \in \R^{m \times r}$ and $Y\subopt \in \R^{n \times r}$.  Based
  on the incomplete set of known entries $\Omega \subset [m] \times [n]$,
  our goal is to recover the matrix $M$, giving rise to the objective
  \begin{equation*}
    \minimize_{X\in \R^{m \times r},Y\in \R^{n \times r}}
    \frac{1}{|\Omega|}
    \sum_{(i,j)\in \Omega} f(x_i, y_j; M_{i,j})
    ~~~ \mbox{where} ~~
    f(x, y; z)
    \defeq | \<x, y\> -  z |
  \end{equation*}
  and $X_i$ and $Y_j$ are the $i$th and $j$th rows of $X$ and $Y$.
  This is the composition of $h(z) = |z|$ and $c(x, y, z)
  = \<x, y\> - z$, so that $f = h \circ c$ is $1$-weakly convex.
\end{example}

\paragraph{Truncated models:}

The prox-linear model~\eqref{eqn:prox-linear-model} may be challenging to
implement for complex compositions (e.g., deep learning),
and it requires a common but
potentially restrictive structure. If instead we know a lower bound on
$f$, we may incorporate this to model
\begin{equation}
  \label{eqn:trunc-model}
  f_x(y; \statval) \defeq \max\left\{
  f(x; \statval) + \<f'(x;\statval), y - x\>,
  \inf_{z \in \mc{X}} f(z; \statval)\right\}.
\end{equation}
In most of our examples---linear and logistic regression, phase retrieval,
matrix completion (and more generally, all typical loss functions in machine
learning)---we have $\inf_{z} f(z; \statval) = 0$. The assumption
that we have a lower bound is thus rarely restrictive. This model
satisfies the conditions~\ref{cond:convex-model}--\ref{cond:subgrad-model},
also satisfying the additional condition
\begin{enumerate}[label=(C.\roman*),leftmargin=*]
  \setcounter{enumi}{3}
\item \label{cond:lower-by-optimal}
  [Lower optimality]
  For all $\statval \in \statdomain$, the models $f_x(\cdot; \statval)$ satisfy
  \begin{equation*}
    f_x(y; \statval) \ge \inf_{z \in \mc{X}}
    f(z; \statval).
  \end{equation*}
\end{enumerate}
\noindent
As we show, Condition~\ref{cond:lower-by-optimal} is sufficient to
derive several optimality and stability properties.

% -*- Mode: latex -*- %

\section{Stability and its consequences}
\label{sec:stability}

In our initial study of stability in optimization~\cite{AsiDu18}, we defined
an algorithm as \emph{stable} if its iterates remain bounded, then showed
several consequences of this in convex optimization (which we review
presently).  A weakness, however, was that the only algorithms whose
stability we were able to demonstrate were very close approximations to
stochastic proximal point methods. Here, we develop two important
extensions. First, we show that any model satisfying
Condition~\ref{cond:lower-by-optimal} has stable iterates under mild
assumptions, in strong contrast to models (e.g.\ linear) that fail the
condition.  Second, we develop an analogous stability theory for weakly
convex functions, proving that accurate enough models are stable.  In
parallel to the convex case, stability suffices for more: it implies
convergence (with an asymptotic rate) to stationary points for any
model-based method on weakly convex functions.

%In~\cite{AsiDu18}, we introduced two notions of stability for optimization methods, 
%which requires boundedness of the iterates, and argued for the importance of this 
%notion for convex functions; in particular, algorithms which are stable are guaranteed 
%to converge in a broader range of problems. In this section, we carry the same task for 
%weakly convex functions and prove that stability is still important in this setting. 
%First, we prove that several models satisfy our notion of stability. 
%Then we show that stable models enjoy better convergence guarantees, 
%in contrast to models which are not stable such 
%as the stochastic subgradient method.

Let us formalize stability.  A pair $(\mc{F}, \mc{P})$ a \emph{collection of
  problems} if $\mc{P}$ consists of probability measures on a sample
space $\statdomain$ and $\mc{F}$ of functions $f : \mc{X}
\times \statdomain \to \R$.
%% let $\mc{A}$ denote the set of positive stepsize
%% sequences $\{\stepsize_k\}$ with $\sum_k \stepsize_k^2 < \infty$.

%%  We make the following definition.
\begin{definition}
  \label{definition:stability}
  An algorithm generating iterates $x_k$ according to the model-based
  update~\eqref{eqn:model-iteration} is \emph{stable in probability} for the
  problems $(\mc{F}, \mc{P})$ if for all $f \in \mc{F}$,
  square-summable positive stepsize sequences $\{\stepsize_k\}$,
  and $P \in \mc{P}$ defining $F(x) = \E_P[f(x; \statrv)]$
  and $\mc{X}\opt = \argmin_{x \in \mc{X}} F(x)$,
  \begin{equation}
    \sup_k \dist(x_k, \mc{X}\opt) < \infty
    ~~ \mbox{with~probability~1}.
    \label{eqn:bounded-distance}
  \end{equation}
  %\end{subequations}
\end{definition}
Standard models, such as the linear model~\eqref{eqn:dumb-linear-model}
and consequent
subgradient method, are unstable~\cite[Sec. 3]{AsiDu18}.
They may even cause super-exponential divergence:

\begin{example}[Divergence]
  \label{ex:exp-divergence}
  Let $F(x) = e^{x} + e^{-x}$, $p < \infty$, and $\stepsize_0 > 0$, and let
  $\stepsize_k$ be any sequence satisfying $\stepsize_k \ge \stepsize_0
  k^{-p}$.  Let $x_{k+1} = x_k - \stepsize_k F'(x_k) = x_k -
  \stepsize_k(e^{x_k} - e^{-x_k}) $ be generated by the gradient method.
  Whenever $x_1$ is large enough, then $\log \frac{x_{k+1}}{x_k}
  \ge 2^k$ for all $k$.
\end{example}

\subsection{The importance of stability in stochastic convex optimization}

To set the stage for what follows, we begin by motivating the importance of
stable procedures.  Briefly, any stable \aProx model converges for any
convex function under weak assumptions, which we now elucidate.
First, we make an
\begin{assumption}
  \label{assumption:weak-second-moment}
  There exists an increasing function $\growfunc : \R_+ \to
  \openright{0}{\infty}$ such that for all $x \in \mc{X}$ and each selection
  $f'(x; \statval) \in \partial f(x; \statval)$,
  \begin{equation*}
    \E\left[\ltwo{f'(x; \statrv)}^2\right] \le \growfunc(\ltwo{x}).
  \end{equation*}
\end{assumption}
\noindent
Assumption~\ref{assumption:weak-second-moment} is relatively weak, and is
equivalent to assuming that $\E[\ltwos{f'(x;\statrv)}^2]$ is bounded on
compacta; it allows arbitrary growth---exponential, super-exponential---in
the norm of the subgradients, just requiring that the second moment exists.

\begin{corollary}[\citet{AsiDu18}, Prop.~1]
  \label{cor:convergence-from-boundedness-convex}
  Assume that $f(\cdot; \statval)$ is convex for each $\statval \in
  \statdomain$ and let Assumption~\ref{assumption:weak-second-moment}
  hold.
  Let the iterates $x_k$ be
  generated by any method satisfying
  Conditions~\ref{cond:convex-model}--\ref{cond:subgrad-model}, and
  additionally assume that with probability 1,
  $\sup_k \norm{x_k} < \infty$. Then $\sum_k \stepsize_k (F(x_k) - F(x\opt))
  < \infty$. If in addition
  $\inf_{x \in \mc{X}} \left\{F(x) - F(x\opt) \mid \dist(x, \mc{X}\opt)
  \ge \epsilon\right\} > 0$
  for all $\epsilon > 0$, then $\dist(x_k, \mc{X}\opt) \cas 0$.
\end{corollary}
\noindent
Corollary~\ref{cor:convergence-from-boundedness-convex} establishes convergence
of stable procedures, and also (via Jensen's inequality) provides
asymptotic rates of convergence for weighted averages
$\sum_k \stepsize_k x_k / \sum_k \stepsize_k$.

Additionally, when the functions $f$ are smooth, any \aProx method achieves
asymptotically optimal convergence whenever the iterates remain bounded.  In
particular, let us assume that $F$ is $\mc{C}^2$ near $x\opt =
\argmin_{\mc{X}} F(x)$ with $\nabla^2 F(x\opt) \succ 0$, and that on some
neighborhood of $x\opt$ the functions $f(\cdot; \statval)$ have
$\lipgrad(\statval)$-Lipschitz gradient with $\E[\lipgrad(\statrv)^2] <
\infty$. We have
\begin{corollary}[\citet{AsiDu18}, Theorem~2]
  \label{corollary:asymptotic-normality}
  In addition to the conditions of
  Corollary~\ref{cor:convergence-from-boundedness-convex}, let the
  conditions of the previous paragraph hold. Then
  $\wb{x}_k = \frac{1}{k} \sum_{i = 1}^k x_i$ satisfies
  \begin{equation*}
    \sqrt{k} (\wb{x}_k - x\opt) \cd \normal\big(0,
    \nabla^2 F(x\opt)^{-1} \cov(\nabla f(x\opt; \statrv))
    \nabla^2 F(x\opt)^{-1} \big).
  \end{equation*}
\end{corollary}
\noindent
This convergence is optimal for any method given samples $\statrv_1, \ldots,
\statrv_k \simiid P$ (see~\cite{DuchiRu19}), and
Corollary~\ref{corollary:asymptotic-normality} highlights the importance of
stability: if any \aProx method is stable, it enjoys (asymptotically)
optimal convergence.

\subsection{Stability of lower-bounded models for convex functions}

With these consequences of stability in hand---convergence and asymptotic
optimality---it behooves us to provide conditions sufficient to
guarantee stability. To that end, we show that lower bounded models
satisfying Condition~\ref{cond:lower-by-optimal} are stable in probability
(Def.~\ref{definition:stability}) for functions whose (sub)gradients grow at
most polynomially.
We begin by stating our polynomial growth assumption.

\begin{assumption}
  \label{assumption:polynomial-growth}
  There exist $C, p<\infty$ such that for every $x \in{\mc{X}}$,
  \begin{equation*}
    \E\left[\ltwo{f'(x; \statrv)}^2\right] \le C(1+\dist(x, \mc{X}\opt)^p),
  \end{equation*}
  and $\E[(f(x\opt;\statrv) - \inf_{z\in \mc{X}}
    f(z;\statrv))^{p/2}] \le C$ for all $x\opt \in
  \mc{X}\opt$.
\end{assumption}	

The analogous condition~\cite{PolyakJu92} for stochastic gradient methods
holds for $p = 2$, or quadratic growth, without which the method may
diverge.  In contrast, Assumption~\ref{assumption:polynomial-growth} allows
polynomial growth; for example, the function $f(x) = x^4$ is permissible,
while the gradient method may exponentially diverge even for stepsizes
$\stepsize_k = 1/k$. The key consequence of
Assumption~\ref{assumption:polynomial-growth} is that if it holds, truncated
models are stable:
\begin{theorem}
  \label{theorem:trunc-stability}
  Assume the function $f(\cdot;\statval)$ is convex for each $\statval \in
  \statdomain$. Let Assumption~\ref{assumption:polynomial-growth} hold and
  $\stepsize_k = \stepsize_0 k^{-\steppow}$ with $\frac{p+2}{p+4} < \steppow
  < 1$.  Let $x_k$ be generated by the iteration~\eqref{eqn:model-iteration}
  with a model satisfying
  Conditions~\ref{cond:convex-model}--\ref{cond:lower-by-optimal}. Then
  \begin{equation*}
    \sup_{k \in \N} \dist(x_k, \mc{X}\opt) < \infty
    \quad  \text{with probability 1.}
  \end{equation*}
  %and $\dist(x_k, \mc{X}\opt)$ converges to some
  %finite value with probability 1.
\end{theorem}

Theorem~\ref{theorem:trunc-stability} shows that truncated methods
enjoy all the
benefits of stability we outline in
Corollaries~\ref{cor:convergence-from-boundedness-convex}
and~\ref{corollary:asymptotic-normality} above. Thus, these
models, whose updates are typically as cheap to compute as
a stochastic gradient step (especially in the common case that
$\inf_z f(z; \statval) = 0$) provide substantial advantage over
methods using only (sub)gradient approximations.

%% Theorem~\ref{theorem:trunc-stability} reveals a major advantage of truncated
%% models over basic stochastic subgradient methods; while SGM can
%% exponentially diverge, the truncated model, which satisfy
%% Condition~\ref{cond:lower-by-optimal}, guarantees that the iterates will
%% remain bounded, regardless of the initial stepsize value.  This boundedness,
%% together with corollary~\ref{cor:convergence-from-boundedness-convex},
%% establishes convergence and asymptotic normality guarantees for the
%% truncated model, demonstrating the substantial advantage that optimization
%% methods achieve by only considering lower bounded models instead of the
%% simple linear model.
%Theorem~\ref{theorem:trunc-stability} demonstrates that Assumption~\ref{assumption:polynomial-growth} is enough
%to prove that the iterates of the truncated model~\eqref{eqn:trunc-model} remain bounded. 
%In contrast, this is not enough for the basic stochastic subgradient methods as seen in our examples.

\subsection{Stability and its consequences for weakly convex functions}

We continue our argument that---if possible---it is beneficial to use more
accurate models, even in situations beyond convexity, investigating the
stability of proximal models~\eqref{eqn:prox-model} for weakly convex
functions. Establishing stability in the weakly convex case requires
a different approach to the convex case, as the iterates may not
make progress toward a fixed optimal set.
In this case, to show stability, we require an assumption
bounding the size of $f'(x; \statrv)$ relative to the population
subgradient $F'$.
\begin{assumption}
  \label{assumption:variance-bounded}
  There exist $C_1, C_2 < \infty$
  such that for all measurable selections $f'(x; \statval) \in \partial
  f(x; \statval)$ and $F'(x) \in \partial F(x)$,
  \begin{equation*}
    \var(f'(x;\statrv))
    \le C_1 \ltwo{F'(x)}^2 + C_2.
  \end{equation*}
\end{assumption}
\noindent
By providing a ``relative'' noise condition on $f'$,
Assumption~\ref{assumption:variance-bounded} allows for a broader class of
functions without global Lipschitz properties (as are typically
assumed~\cite{DavisDr19}), such as the phase retrieval and matrix completion
objectives (Examples~\ref{example:PR} and~\ref{example:matComp}).  It can
allow exponential growth, addressing the challenges in
Ex.~\ref{ex:exp-divergence}. For example, let $f(x; 1) = e^x$ and $f(x; 2) =
e^{-x}$, where $\statrv$ is uniform in $\{1,2\}$ so that $F(x) = \half(e^x +
e^{-x})$; then $\E[f'(x;\statrv)^2] = 2 F'(x)^2 + \half$.

To describe convergence and stability guarantees in non-convex (even
non-smooth) settings, we require appropriate definitions.  Finding global
minima of non-convex functions is computationally
infeasible~\cite{NemirovskiYu83}, so we follow established practice and
consider convergence to stationary points, specifically using the
convergence of the Moreau envelope~\cite{DavisDr19, DrusvyatskiyLe18}.  To
formalize, for $x \in \R^n$ and $\lambda \ge 0$, the \emph{Moreau envelope}
and associated proximal map are
\begin{equation*}
  F_\lambda(x) \defeq \inf_{y \in \mc{X}}
  \Big\{F(y) + \frac{\lambda}{2} \ltwo{y - x}^2\Big\}
  ~~
  \mbox{and} ~~
  \prox_{F/\lambda}(x) \defeq \argmin_{y \in \mc{X}}
  \Big\{F(y) + \frac{\lambda}{2} \ltwo{y - x}^2 \Big\}.
\end{equation*}
For large enough $\lambda$, the minimizer $x^\lambda \defeq
\prox_{F/\lambda}(x)$ is unique whenever $F$ is weakly convex. Adopting
the techniques pioneered by \citet{DavisDr19} for convergence
of stochastic methods on weakly convex problems, our
convergence machinery relies on the Moreau envelope's connections
to (near) stationarity:
\begin{equation}
  \label{eqn:moreau-useful}
  \nabla F_\lambda(x) = \lambda(x - x^\lambda),
  ~~~
  F(x^\lambda) \le F(x),
  ~~~
  \dist(0, \partial F(x^\lambda)) \le \ltwo{\nabla F_\lambda(x)}.
\end{equation}
%% \begin{itemize}
%% \item $\nabla F_\lambda(x) = \lambda(x-x^\lambda)$ and $\ltwo{x^\lambda  - x} = \frac{1}{\lambda} \ltwo{ \nabla F_\lambda(x) }$
%% \item $F(x^\lambda) \le F(x)$
%% \item $\dist(0; \partial F(x^\lambda)) \le \ltwo{\nabla F_\lambda(x) }$
%% \end{itemize}
The three properties~\eqref{eqn:moreau-useful} imply that any nearly
stationary point $x$ of $F_\lambda(x)$---when $\ltwos{\nabla F_\lambda(x)}$
is small---is close to a nearly stationary point $x^{\lambda}$ of the
original function $F(\cdot)$.  To prove convergence for weakly convex
methods, then, it is sufficient to show that $\nabla F_\lambda(x_k) \to 0$.

%% it is sufficient 
%% The goal, then, is to
%% prove convergence to a point $x$ such that $\ltwo{x^\lambda - x}$ is small,
%% as this implies that $x$ is close to a nearly stationary point.

Using full proximal models, it turns out, guarantees convergence.
\begin{theorem}
  \label{theorem:convergence-basic}
  Let Assumption~\ref{assumption:variance-bounded} hold and $\lambda$ be
  large enough that $\E[\weakconvexfunc(\statrv)] < \lambda$, and
  assume that $\inf_{x \in \mc{X}} F(x) > -\infty$ and
  $\E[\weakconvexfunc(\statrv)^2] < \infty$.
  Let $x_k$ be generated by
  the iteration~\eqref{eqn:model-iteration} with the proximal
  model~\eqref{eqn:prox-model}.
  %Assume further that 
  %$\E[\weakconvexfunc(\statrv_k)] < \lambda$ and 
  %$\E[\weakconvexfunc(\statrv_k)^2] < \infty$. 
  Then there exists a random variable $G_\lambda < \infty$ such that
  \begin{equation*}
    F_\lambda(x_k) \cas G_\lambda
  \end{equation*}
  and $\sum_{k = 1}^\infty \stepsize_k \ltwo{\nabla F_\lambda(x_k)}^2
  < \infty$ with probability one.
\end{theorem}

The theorem shows that $F_\lambda(x_k)$ is bounded almost surely. Thus, if
the function $F(\cdot)$ is coercive, meaning $F(x) \uparrow \infty $ as
$\norm{x} \to \infty$, then the Moreau envelope $F_\lambda(\cdot)$ is also
coercive, so that we have the following corollary.
\begin{corollary}
  \label{cor:weakly-convex-stability}
  Let the conditions of Theorem~\ref{theorem:convergence-basic} hold and let
  $F$ be coercive. Then
  \begin{equation*}
    \sup_{k \in \N} \dist(x_k, \mc{X}\opt) < \infty
    ~~~\text{with probability 1.}
  \end{equation*}
\end{corollary}	

In parallel with our devlopment of the convex case, stability is sufficient
to develop convergence results for any model-based \aProx method,
highlighting its importance.  Indeed, we can show that stable methods
guarantee various types of convergence to stationary points, though for
probability one convergence of the iterates, we require a slightly elaborate
assumption~\cite[cf.][]{DuchiRu18c, DavisDrKaLe19}, which
rules out some pathological limiting cases.

\begin{assumption}[Weak Sard]
  \label{assumption:weak-sard}
  Let $\mc{X}\stationary = \{x \mid 0 \in \partial F(x) \}$
  be the collection of stationary points of $F$ over $\mc{X}$.
  The Lebesgue measure of the image $F(\mc{X}\stationary)$ is zero.
\end{assumption}
Under this assumption, \aProx methods converge to stationary 
points whenever the iterates are stable.
\begin{proposition}
  \label{proposition:convergence-from-boundedness-weakly-convex}
  Let Assumption~\ref{assumption:weak-second-moment} hold and the iterates
  $x_k$ be generated by any method satisfying
  Conditions~\ref{cond:convex-model}--\ref{cond:subgrad-model}.  Assume that
  $\lambda$ is large enough that $\E[\weakconvexfunc(\statrv)] <
  \lambda$. There exists a finite random variable $G_\lambda$ such that on
  the event that $\sup_k \ltwo{x_k} < \infty$, we have
  \begin{equation}
    \label{eqn:boundedness-to-stationary-convergence}
    \sum_k \stepsize_k
    \ltwo{\nabla F_\lambda(x_k)}^2 < \infty
    ~~~ \mbox{and} ~~~
    F_\lambda(x_k) \to G_\lambda
    ~~~ \mbox{with~probability~}1.
  \end{equation}
  If additionally Assumption~\ref{assumption:weak-sard}
  holds, then $\ltwo{\nabla F_\lambda(x_k)} \cas 0$ and $\dist(x_k,
  \mc{X}\stationary) \cas 0$.
\end{proposition}

In passing, we note that the finite sum
condition~\eqref{eqn:boundedness-to-stationary-convergence} is enough to
develop a type of conditional $\ell_2$-convergence, which is similar to the
non-asymptotic rates of convergence that stochastic (sub)gradient methods
achieve to stationary points~\cite{DavisDr19, GhadimiLa13}. Indeed, assume
$\stepsize_k = \stepsize_0 k^{-\steppow}$ for some $\steppow \in (\half, 1)$
and that the iterates $x_k$ are stable. Now let $I_k$ be an index chosen from
$\{1, \ldots, k\}$ with probabilities $P(I_k = i) = \stepsize_i / \sum_{j=1}^k
\stepsize_j$. Then
inequality~\eqref{eqn:boundedness-to-stationary-convergence} shows
that
\begin{equation*}
  \limsup_k
  k^{1 - \steppow} \E\left[\ltwo{\nabla F_\lambda(x_{I_k})}^2 \mid \mc{F}_k\right]
  < \infty ~~ \mbox{with~probability~1}.
\end{equation*}
This provides an asymptotic analogue of the convergence rates that
stochastic model-based methods achieve on Lipschitzian
functions~\cite{DavisDr19}.

% -*- Mode: latex -*- %

\section{Fast convergence for easy problems}
\label{sec:aProx-adv-fast-conv}

In many engineering and learning applications, solutions \emph{interpolate}
the data. Consider, for example, signal recovery problems with $b = Ax\opt$,
or modern machine learning applications, where frequently training error is
zero~\cite{LeCunBeHi15, BelkinHsMi18}. We consider such problems here,
showing how models that satisfy the lower bound
condition~\ref{cond:lower-by-optimal} enjoy linear convergence, extending
our earlier results~\cite{AsiDu18} beyond convex optimization. We begin with
a
\begin{definition}
  \label{definition:easy-problems}
  Let $F(x) \defeq \E_P[f(x; \statrv)]$. Then $F$ is \emph{easy to optimize}
  if for each $x\opt \in \mc{X}\opt \defeq \argmin_{x \in \mc{X}} F(x)$ and
  $P$-almost all $\statval \in \statdomain$ we have
  \begin{equation*}
    \inf_{x \in \mc{X}} f(x; \statval) = f(x\opt; \statval).
  \end{equation*}
\end{definition}

For such problems, we can guarantee progress toward minimizers
as long as $f$ grows quickly enough away from $x\opt$, as the following
lemma (generalizing our result~\cite[Lemma~4.1]{AsiDu18}) shows.
\begin{lemma}
  \label{lemma:shared-min-progress}
  Let $F$ be easy to optimize (Definition~\ref{definition:easy-problems}).
  Let $x_k$ be generated by the updates~\eqref{eqn:model-iteration} using a
  model satisfying
  Conditions~\ref{cond:convex-model}--\ref{cond:lower-by-optimal}.  Then for
  any $x\opt \in \mc{X}\opt$,
  \begin{equation*}
    \ltwo{x_{k+1} - x\opt}^2
    \le (1 + \stepsize_k \weakconvexfunc(\statrv_k))
    \ltwo{x_k - x\opt}^2
    - [f(x_k; \statrv_k) - f(x\opt; \statrv_k)]\min\left\{
    \stepsize_k,
    \frac{f(x_k; \statrv_k) - f(x\opt; \statrv_k)}{
      \ltwo{f'(x_k; \statrv_k)}^2}\right\}.
  \end{equation*}
\end{lemma}

\noindent
Lemma~\ref{lemma:shared-min-progress} allows us to prove fast convergence
so long as $f$ grows quickly enough away from $x\opt$; a sufficient
condition for us is a so-called \emph{sharp growth} condition away
from the optimal set $\mc{X}\opt$. To meld with the progress guarantee
in Lemma~\ref{lemma:shared-min-progress}, we consider the following
assumption. 
\begin{assumption}[Expected sharp growth]
  \label{assumption:expected-sharp-growth}
  There exist constants $\lambda_0, \lambda_1 > 0$ such that
  for all $\stepsize \in \R_+$ and $x \in \mc{X}$
  and $x\opt \in \mc{X}\opt$,
  \begin{align*}
    \E\left[
      \min\left\{\stepsize, \frac{f(x; \statrv) - f(x\opt; \statrv)}{
	\ltwos{f'(x; \statrv)}^2} \right\}
      (f(x;\statrv) - f(x\opt; \statrv)) \right] 
    \ge \dist(x, \mc{X}\opt) \min\left\{\lambda_0 \stepsize,
    \lambda_1 \dist(x, \mc{X}\opt)\right\}.
  \end{align*}
\end{assumption}

Assumption~\ref{assumption:expected-sharp-growth} is perhaps
too-tailored to Lemma~\ref{lemma:shared-min-progress} for
obvious use, so we discuss a few situations where it holds.
One simple sufficient condition is the
small-ball-type condition that
there exists $C < \infty$ such that
$\P(f(x; \statrv) - f(x\opt; \statrv) \ge \epsilon \dist(x, \mc{X}\opt))
\ge 1 - C \epsilon$ for $\epsilon > 0$ and that
$\E[\ltwo{f'(x; \statrv)}^2] \le C(1 + \dist(x, \mc{X}\opt)^2)$.
We can also be more explicit:

\begin{example}[Phase retrieval, example~\ref{example:PR} continued]
  \label{example:finite-sample-pr}
  Consider the (real-valued) phase retrieval problem with objective
  $f(x;(a,b)) = | \<a,x\>^2 - b|$.  Let us assume the vectors $a_i \in \R^n$
  are drawn from a distribution satisfying the small ball condition
  $P(|\<a_i,u\>| \ge \epsilon \ltwo{u}) \ge 1-\epsilon$ for $\epsilon > 0$
  and any $u \in \R^n$, and additionally that $\E[\ltwo{a_i}^2] \le M^2 n$
  and $\E[ \<a_i,x\>^2] \le M^2 \ltwo{x}^2$ for some $M < \infty$. For each
  $i$, define the events
  \begin{align*}
    E_{1,i}(x)
    & \defeq \left\{ \left|\<a_i,x\>^2 - \<a_i,x\opt\>^2 \right| \ge \epsilon^2
    \dist(x,\mc{X}\opt) \sqrt{\ltwo{x}^2 + \ltwo{x\opt}^2 } \right\},
    \\ E_{2,i}(x) &
    \defeq \left\{ \ltwo{a_i}^2 \le 8 M^2n\right\},
    ~~~
    E_{3,i}(x) \defeq \left\{\<a_i,x\>^2 \le 8 M^2\ltwo{x}^2 \right\}
    %% = \Big\{\big\<a_i, \frac{x}{\ltwo{x}}\big \> \in
    %% [\pm 2 \sqrt{2} M] \Big\}.
  \end{align*}
  The small ball condition implies $P(E_{1,i}(x)) \ge 1 - 2 \epsilon$,
  and Markov's inequality gives $P(E_{2,i}(x)) \ge \frac{7}{8}$ and $P(E_{3,i}(x))
  \ge \frac{7}{8}$ for all $x \in \R^n$.
  Letting $E_i(x) = 1$ if $E_{1,i}(x), E_{2,i}(x), E_{3,i}(x)$ each occur
  and $E_i(x) = 0$ otherwise,
  a VC-dimension calculation (see, for example,
  \cite[Appendix~A]{DuchiRu18a}) thus implies that
  with probability at least $1 - e^{-t}$,
  we have
  \begin{equation}
    \label{eqn:dataset-for-pr}
    \inf_{x \in \R^n} \frac{1}{m}
    \sum_{i = 1}^m E_i(x)
    \ge \frac{3}{4} - 2 \epsilon - C \sqrt{\frac{n + t}{m}}.
  \end{equation}
  
  Let $\epsilon = \frac{1}{8}$ for simplicity. Then
  evidently with high probability over the draw of a random
  data matrix $A \in \R^{m \times n}$ and $b = |Ax\opt|^2$,
  the rows $a_i$
  satisfy
  condition~\eqref{eqn:dataset-for-pr}, so we have
  \begin{align*}
    \lefteqn{\frac{1}{m} \sum_{i = 1}^m
      \min\left\{\stepsize, \frac{f(x; (a_i, b_i)) - f(x\opt; (a_i, b_i))}{
        \ltwo{f'(x; (a_i, b_i))}^2}\right\}
      \left[f(x; (a_i, b_i)) - f(x\opt; (a_i, b_i))\right]} \\
    & \ge
    \frac{1}{m} \sum_{i : E_i(x) = 1}
    \min\left\{\stepsize, \frac{\epsilon^2 \dist(x, \mc{X}\opt)
      \sqrt{\ltwo{x}^2 + \ltwo{x\opt}^2}}{
      64 M^4 n \ltwo{x}^2}\right\}
    \epsilon^2 \dist(x, \mc{X}\opt)
    \sqrt{\ltwo{x}^2 + \ltwo{x\opt}^2} \\
    & \ge c \dist(x, \mc{X}\opt) \cdot \min\left\{ \ltwo{x\opt} \stepsize,
    \frac{c}{M^4 n} \dist(x, \mc{X}\opt)\right\}
  \end{align*}
  for a numerical constant $c > 0$.
  On this finite sample of size $m$,
  Assumption~\ref{assumption:polynomial-growth} holds
  for the empirical objective $F(x) = \frac{1}{m}
  \sum_{i = 1}^m f(x; (a_i, b_i))$
  with $\lambda_0 = c \ltwo{x\opt}$,
  and $\lambda_1 = \frac{c}{16M^4 n}$.
\end{example}

The following proposition is our main result
in this section, showing that lower
bounded models enjoy linear convergence on easy problems.
%% The result for the convex setting appeared in our
%% earlier paper~\cite{AsiDu18}.
\begin{proposition}
  \label{proposition:sharp-growth-convex-convergence}
  Let Assumption~\ref{assumption:expected-sharp-growth} hold and $x_k$ be
  generated by the stochastic iteration~\eqref{eqn:model-iteration} using
  any model satisfying
  Conditions~\ref{cond:convex-model}--\ref{cond:lower-by-optimal}, where the
  stepsizes $\stepsize_k$ satisfy $\stepsize_k = \stepsize_0 k^{-\beta}$ for
  some $\beta \in (0, 1)$.
  If $f(\cdot;\statrv_k)$ is $\weakconvexfunc(\statrv_k)$-weakly
  convex with $\E [\weakconvexfunc(\statrv_k)] = \wb{\weakconvexfunc}$,
  then for any $m \in \N$ and $\epsilon > 0$, there exists a finite random
  variable $V_{\infty,m} < \infty$ such that
  \begin{equation*}
    \frac{\dist(x_k, \mc{X}\opt)^2}{(1 - \lambda_1)^k}
    \cdot \indic{\max_{m \le i \le k-1} \dist(x_i, \mc{X}\opt)
      \le \frac{\lambda_0}{(1 + \epsilon) \wb{\weakconvexfunc}}}
    \cas V_{\infty, m}.
  \end{equation*}
\end{proposition}

When the functions $f$ are convex, we have $\wb{\weakconvexfunc} = 0$,
so that Proposition~\ref{proposition:sharp-growth-convex-convergence}
guarantees linear convergence for easy problems.
In the case that $\wb{\weakconvexfunc} > 0$,
the result is conditional: \emph{if} an \aProx method converges to one of
the sharp minimizers of $f$, then this convergence is linear
(i.e.\ geometrically fast). In the case of
phase retrieval, we can guarantee
convergence:

\begin{example*}[Phase retrieval, Example~\ref{example:finite-sample-pr}
    continued] With $F(x) = \frac{1}{m} \lone{|Ax|^2 -
    |Ax\opt|^2}$, the conditions in
  Example~\ref{example:finite-sample-pr} guarantee that with
  high probability
  $F(x) \ge \lambda \ltwo{x - x\opt} \ltwo{x + x\opt}$ for a numerical
  constant $\lambda > 0$ and that $F$ is $\weakconvexfunc$-weakly convex for
  $\weakconvexfunc$ a numerical constant~\cite{DuchiRu18a}. In this
  case, an integration argument~\cite{DavisDrMaPa18}
  guarantees that
  \begin{equation*}
    \<F'(x), x - x\opt\> \ge \lambda \ltwo{x - x\opt} \ltwo{x + x\opt}
    - \frac{\weakconvexfunc}{2} \ltwo{x - x\opt}^2.
  \end{equation*}
  If $\ltwo{x - x\opt} \le \ltwo{x\opt}$, then $\ltwo{x + x\opt} \ge
  \ltwo{x\opt}$, so $\<F'(x), x - x\opt\> > 0$ whenever $\ltwo{x - x\opt} <
  \frac{2 \lambda}{\weakconvexfunc} \ltwo{x\opt}$. Thus $F$ has no
  stationary points in the set $\{x : \ltwo{x - x\opt} <
  \frac{2 \lambda}{\weakconvexfunc} \ltwo{x\opt}\}$. This set is
  unavailable; however, it is possible to use a
  spectral algorithm to construct an initializer $x_0$ satisfying $\ltwo{x_0
    - x\opt} < \frac{\lambda}{\weakconvexfunc} \ltwo{x\opt}$ and estimate
  $\ltwo{x\opt}$ to high accuracy as soon as the number of measurements $m
  \gtrsim n$~\cite{DuchiRu18a}.  By defining 
  $\mc{X} \defeq \{x : \ltwo{x - x_0} \le \frac{\lambda}{\weakconvexfunc}
  \ltwo{x\opt}\}$, we then guarantee (via
  Proposition~\ref{proposition:convergence-from-boundedness-weakly-convex})
  that the iterates $x_k$ converge to $x\opt$, and
  Proposition~\ref{proposition:sharp-growth-convex-convergence} implies
  that for a numerical constant $c > 0$,
  there is a (random) $B < \infty$ such that
  $\ltwo{x_k - x\opt}^2 \le B \cdot (1 - c/n)^k$ eventually.
  To achieve an $\epsilon$-accurate solution to the (robust) phase retrieval
  objective, the truncated stochastic model~\eqref{eqn:trunc-model} requires
  $O(n)$ operations per iteration and at most $O(n \log \frac{1}{\epsilon})$
  steps. Combined with the spectral initialization, which requires
  time $O(mn \log \frac{1}{\epsilon})$, the
  overall computation time is at most $O(m n \log
  \frac{1}{\epsilon})$, which is the best known for phase retrieval.
\end{example*}

% -*- Mode: latex -*- %

\newcommand{\numTests}{T}
\newcommand{\numIterations}{K}

\newcommand{\dimExp}{n} %dimension of the data
\newcommand{\numSamplesExp}{m} %number of samples
\newcommand{\MatExp}{A} %measurement matrix
\newcommand{\matExp}{a} %denote a row of the matrix
\newcommand{\vecExp}{b} %vector b of measurmens
\newcommand{\noiseMagExp}{\sigma} %noise magnitude
\newcommand{\noiseVecExp}{v} %noise vector

%variables for the matrix completion setting
\newcommand{\matCompMat}{M} 
\newcommand{\matCompVara}{X} 
\newcommand{\matCompVarb}{Y}
\newcommand{\numRows}{m} 
\newcommand{\numCols}{n} 
\newcommand{\dimVar}{r} 

%variables for logistic regression
\newcommand{\varLogReg}{\theta} 
\newcommand{\vecLogReg}{x} 
\newcommand{\labelLogReg}{y} 

%variables for the multi class hinge loss
\newcommand{\vecHinge}{a} 
\newcommand{\labelHinge}{\ell} 
\newcommand{\numClassesHinge}{K} 
\newcommand{\flipProp}{p}

\section{Experiments}
\label{sec:exp}

An important question in the development of any optimization method is its
sensitivity to algorithm parameters. Consequently, we conclude by
experimentally examining convergence time and robustness of each of our
optimization methods.  We consider each of the models we investigate in the
paper: the stochastic gradient method (i.e.\ the linear
model~\eqref{eqn:dumb-linear-model}), the proximal
model~\eqref{eqn:prox-model}, the prox-linear
model~\eqref{eqn:prox-linear-model},
and the (lower) truncated model~\eqref{eqn:trunc-model}.

%% the following models, which we explored
%% earlier in the paper.
%% \begin{enumerate}[label=(\roman*),leftmargin=*]
%% 	\setlength{\itemsep}{0pt}
%% 	\item Stochastic gradient method (SGM): uses the
%% 	linear model~\eqref{eqn:dumb-linear-model}.
%% 	\item Proximal: uses the full
%% 	model~\eqref{eqn:prox-model}.
%% 	\item Prox-linear: uses the prox-linear model~\eqref{eqn:prox-linear-model}.
%% 	\item Truncated: uses the lower truncated model~\eqref{eqn:trunc-model}.
%% 	%\item Bundle: uses the bundle (cutting plane) model~\eqref{eqn:bundling}
%% 	%with two lines, that
%% 	%is, with $i = 1$.
%% \end{enumerate}

We test both convergence time and, dovetailing
with our focus in this paper,
robustness to stepsize specification for several problems:
phase retrieval
(Section~\ref{section:PR}), matrix completion
(Section~\ref{section:MatComp}), and two classification problems using
deep learning (Section~\ref{section:NN}). We consider stepsize sequences
of the form $\stepsize_k = \stepsize_0 k^{-\steppow}$, where $\steppow \in
(1/2, 1)$, and perform $\numIterations$ iterations over a
wide range of different initial stepsizes $\stepsize_0$. (For brevity,
we present
results only for the stepsize $\steppow = .6$; our other experiments
with varied $\steppow$ were similar.) For a fixed
accuracy $\epsilon > 0$, we record the number of steps $k$ required to
achieve $F(x_k) - F(x\opt) \le \epsilon$, reporting these times (where we
terminate each run at iteration $\numIterations$).  We perform $\numTests$
experiments for each initial stepsize choice, reporting the median
time to $\epsilon$-accuracy and 90\% confidence
intervals.

%% We perform this experiment for a number of different common and
%% important optimization problems.

\subsection{Phase Retrieval}
\label{section:PR}

\begin{figure}[ht]
  \begin{center}
    \begin{tabular}{cc}
      \hspace{-.25cm}
      \begin{overpic}[width=.5\columnwidth]{%,grid]{%
	  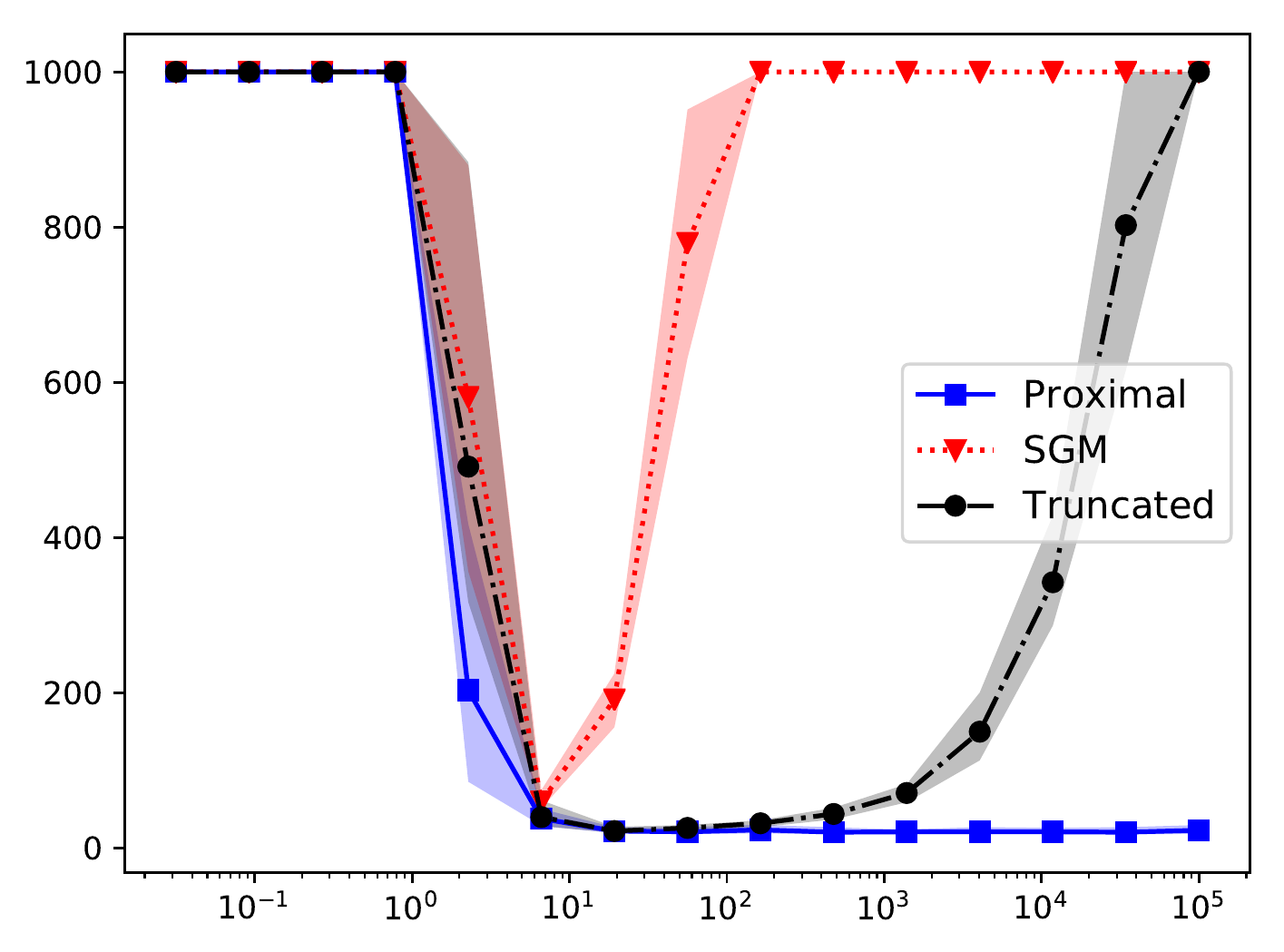}
	\put(35,-1){{\small Initial stepsize $\stepsize_0$}}
	\put(-3,15){
	  \rotatebox{90}{{\small Time to accuracy $\epsilon = .05$}}}
      \end{overpic} &
      \hspace{-.5cm}
      \begin{overpic}[width=.5\columnwidth]{%,grid]{%
	  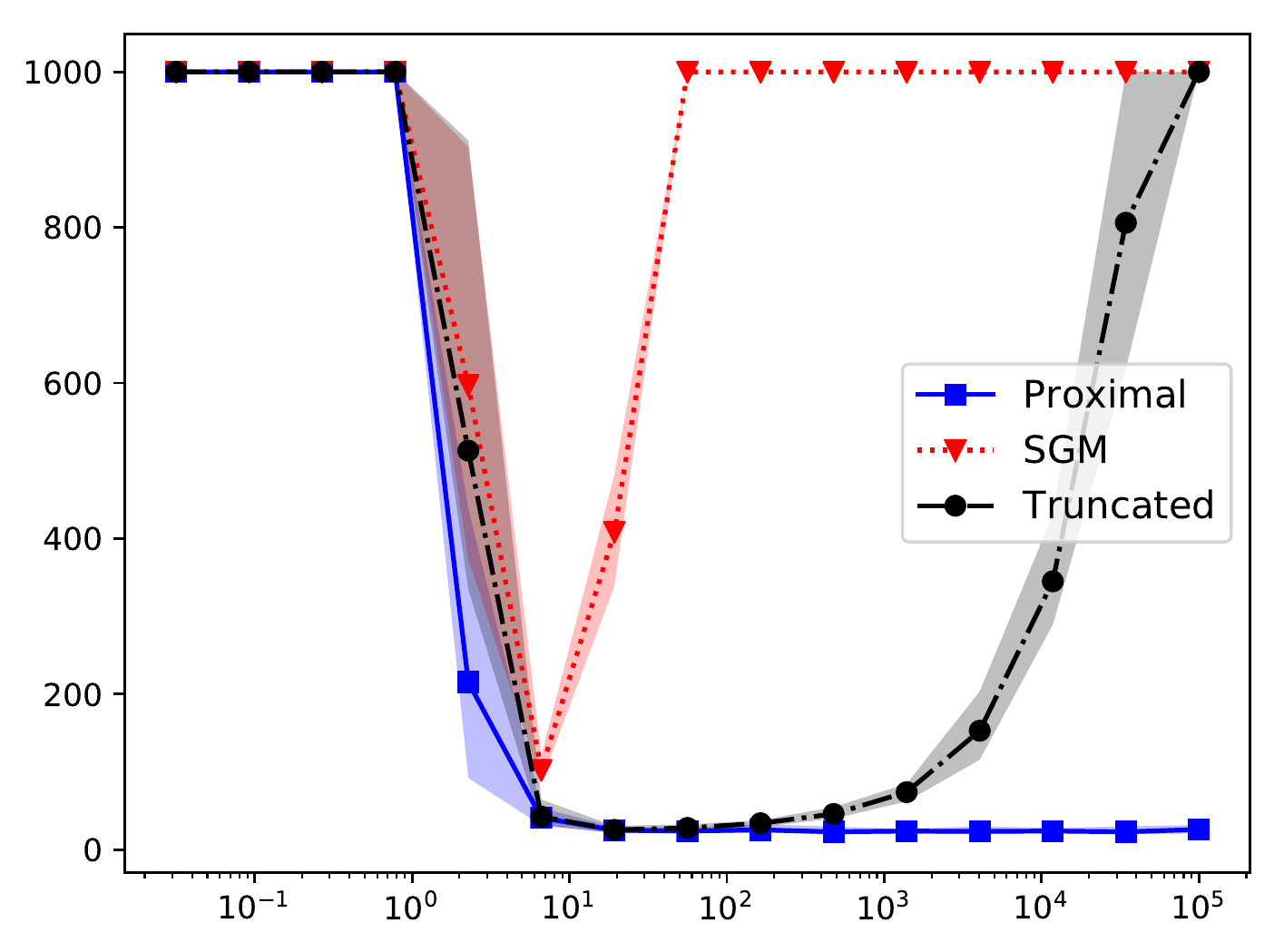}
	\put(35,-1){{\small Initial stepsize $\stepsize_0$}}
	\put(-3,15){
	  \rotatebox{90}{{\small Time to accuracy $\epsilon = .05$}}}
      \end{overpic} \\
      (a) & (b)
    \end{tabular}
    \caption{\label{fig:PR} The number of iterations to achieve
      $\epsilon$-accuracy as a function of the initial step size
      $\stepsize_0$ for phase retrieval with $\dimExp = 50$,
      $\numSamplesExp = 1000$.
    }
  \end{center}
\end{figure}

We start our experiments with the phase retrieval problem in
Examples~\ref{example:PR} and~\ref{example:finite-sample-pr}, focusing on
the real case for simplicity, where we are given $A \in \R^{m \times n}$
with rows $\matExp_i \in \R^\dimExp$ and $b = (Ax\opt)^2 \in
\R_+^\numSamplesExp$ for some $x\opt \in \R^\dimExp$. Our objective is
the non-convex and non-smooth function
\begin{equation*}
  F(x) = \frac{1}{\numSamplesExp}
  \sum_{i=1}^{\numSamplesExp} \left| \< \matExp_i, x \>^2 - \vecExp_i \right|.
\end{equation*}
In our experiments, we sample the entries of the vectors $a_i$ and $x\opt$ 
i.i.d.\ $\normal(0,I_\dimExp)$. 

We present the results of this experiment in Figure~\ref{fig:PR}, comparing
the stochastic gradient method~\eqref{eqn:dumb-linear-model}, proximal
method~\eqref{eqn:prox-model}, and truncated method~\eqref{eqn:trunc-model}
(which yields updates identical to the prox-linear
model~\eqref{eqn:prox-linear-model} in this case). The plots demonstrate the
expected result that the stochastic gradient method has good performance in
a narrow range of stepsizes, $\stepsize_1 \approx 10$ in this case, while
better approximations for \aProx yield better convergence over a large range
of stepsizes.  The truncated model~\eqref{eqn:trunc-model} exhibits some
oscillation for large stepsizes, in contrast to the exact
model~\eqref{eqn:prox-model}, which is robust to all stepsizes $\stepsize_0
\ge 10$.

\subsection{Matrix Completion}
\label{section:MatComp}
\begin{figure}[ht]
  \begin{center}
    \begin{tabular}{cc}
      \begin{overpic}[width=.5\columnwidth]{%,grid]{%
	  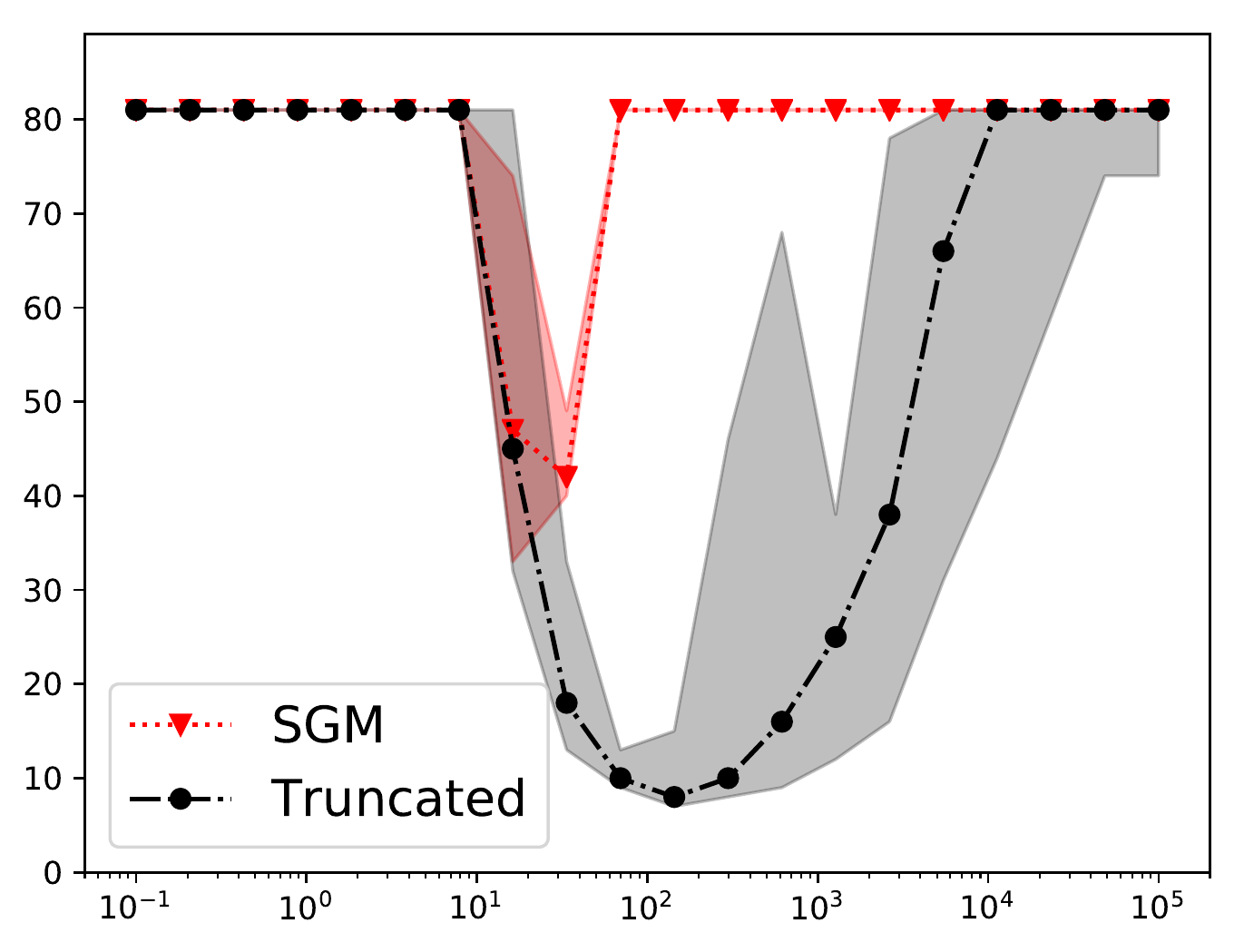}
	\put(35,-1){{\small Initial stepsize $\stepsize_0$}}
	\put(-5,15){
	  \rotatebox{90}{{\small Time to accuracy $\epsilon = .05$}}}
      \end{overpic} &
      \begin{overpic}[width=.5\columnwidth]{%,grid]{%
	  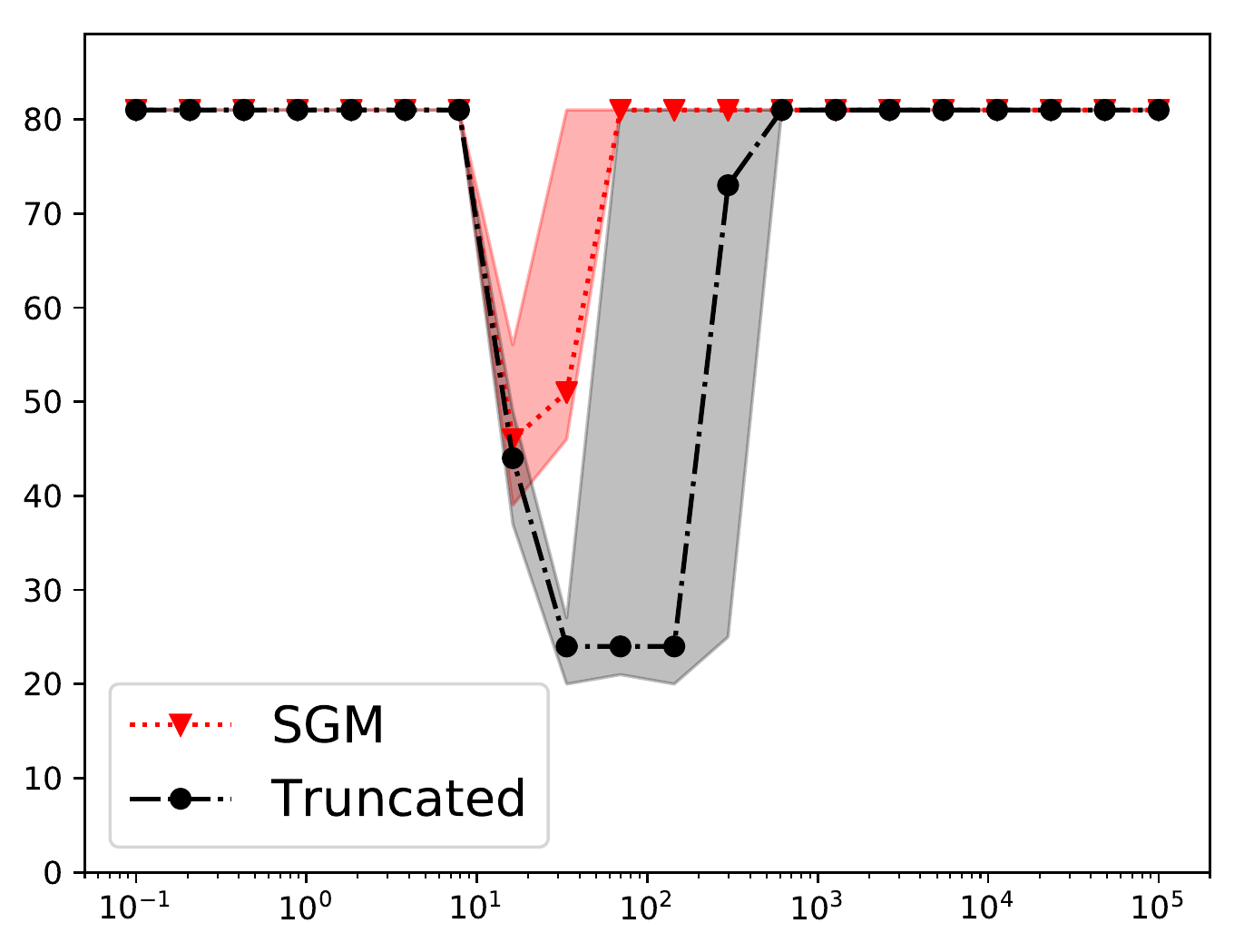}
	\put(35,-1){{\small Initial stepsize $\stepsize_0$}}
	\put(-5,15){
	  \rotatebox{90}{{\small Time to accuracy $\epsilon = .05$}}}
      \end{overpic} \\
      (a) & (b)
    \end{tabular}
    \caption{\label{fig:matComp} Number of iterations to achieve
      $\epsilon$-accuracy as a function of initial step size
      $\stepsize_0$ for matrix completion with $\numRows = 2000$, $\numCols =
      2400$, $\dimVar=5$. Estimated ranks (a) $\hat{\dimVar} = 5$ and (b)
      $\hat{\dimVar} = 10$.}
  \end{center}
\end{figure}

For our second experiment, we investigate the performance of \aProx
procedures for the matrix completion problem of
Example~\ref{example:matComp}.  In this setting, we are given a matrix
$\matCompMat \matCompVara\subopt \matCompVarb\subopt^T $, for
$\matCompVara\subopt \in \R^{\numRows \times \dimVar }$ and
$\matCompVarb\subopt\in \R^{\numCols \times \dimVar }$.  We are also given a
set of indices $\Omega \subset [m] \times [n]$ where the value of matrix
$\matCompMat$ is known. We aim to recover $\matCompMat$ while our access to
$\matCompMat$ is restricted only to indices in the set $\Omega$, so our
goal is to minimize
\begin{equation*}
  F( \matCompVara, \matCompVarb) =
  \frac{1}{\left|\Omega\right|}  \sum_{i,j \in{\Omega}}
  \left|  \matCompVara_i^T  \matCompVarb_j - \matCompMat_{i,j} \right|.
\end{equation*}
Here the matrices $\matCompVara \in \R^{\numRows \times \hat \dimVar}$ and
$\matCompVarb \in \R^{\numCols \times \hat \dimVar}$, where the estimated
rank $\hat{\dimVar} \ge \dimVar$.  We generate the data by drawing the
entries of $\matCompVara\subopt$ and $\matCompVarb\subopt$
i.i.d.\ $\normal(0, 1)$ and choosing $\Omega$ uniformly at random of size
$|\Omega| = 5(nr + mr)$.
We present the timing results in Figure~\ref{fig:matComp}, which tells a
similar story to Figure~\ref{fig:PR}: better approximations, such as the
truncated models (which again yield identical updates to the prox-linear
models~\eqref{eqn:prox-linear-model}), are significantly more robust to
stepsize specification. In this case, the full proximal update requires
solving a nontrivial quartic, so we omit it.

\subsection{Neural Networks}
\label{section:NN}

\begin{figure}[ht]
  \begin{center}
    \begin{tabular}{cc}
      \begin{overpic}[width=.5\columnwidth]{%,grid]{%
	  %../../Theory/paper-non-convex/plots/NN/CIFAR_minLoss.pdf}
	  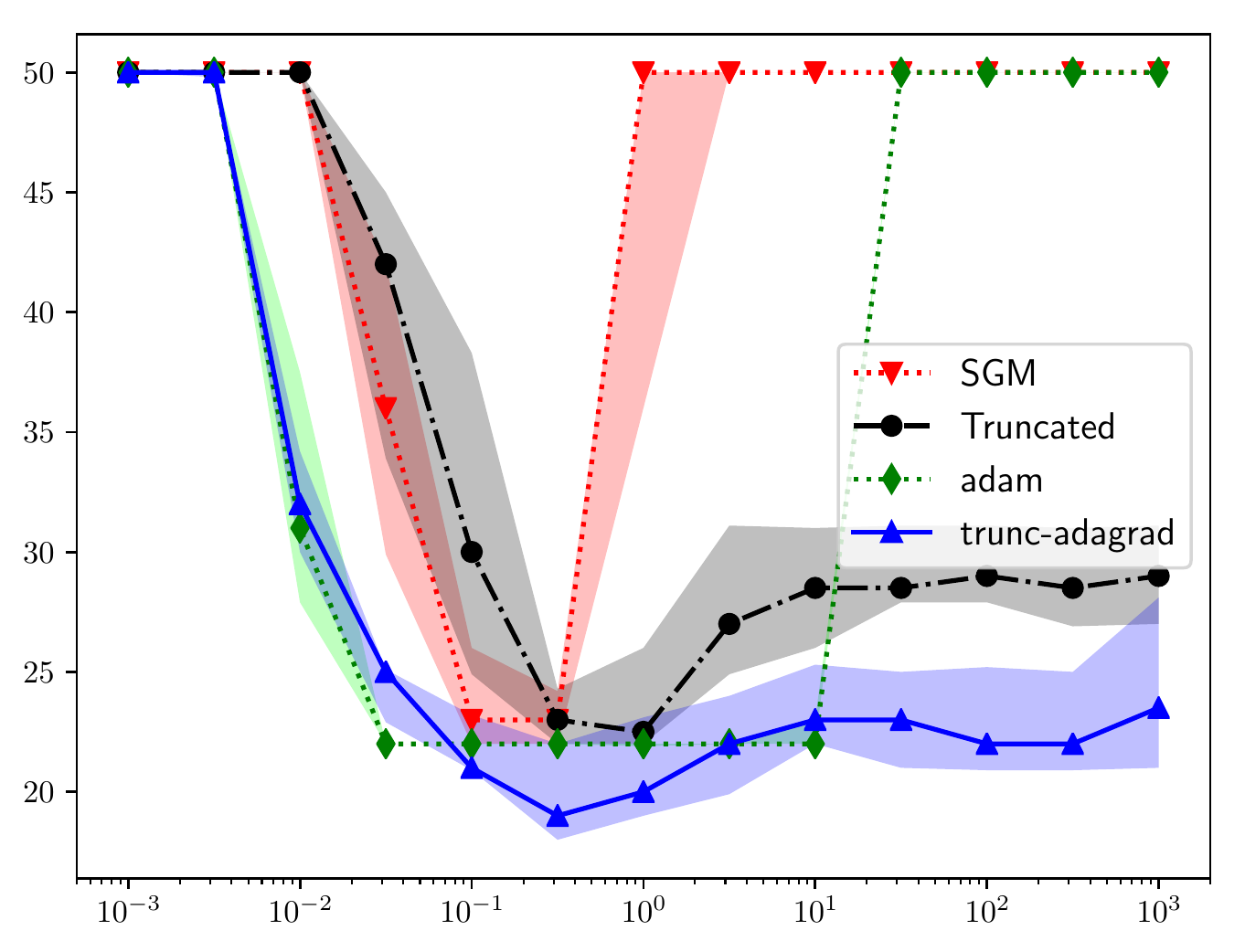}
	\put(35,-1){{\small Initial stepsize $\stepsize_0$}}
	\put(-5,15){
	  \rotatebox{90}{{\small Time to error $\epsilon = .11$}}}
	%\put(33,72){
	%	\tikz{\path[draw=white,fill=white] (0, 0) rectangle (3cm,.35cm);}
	%}
      \end{overpic} &
      \begin{overpic}[width=.5\columnwidth]{%,grid]{%
	  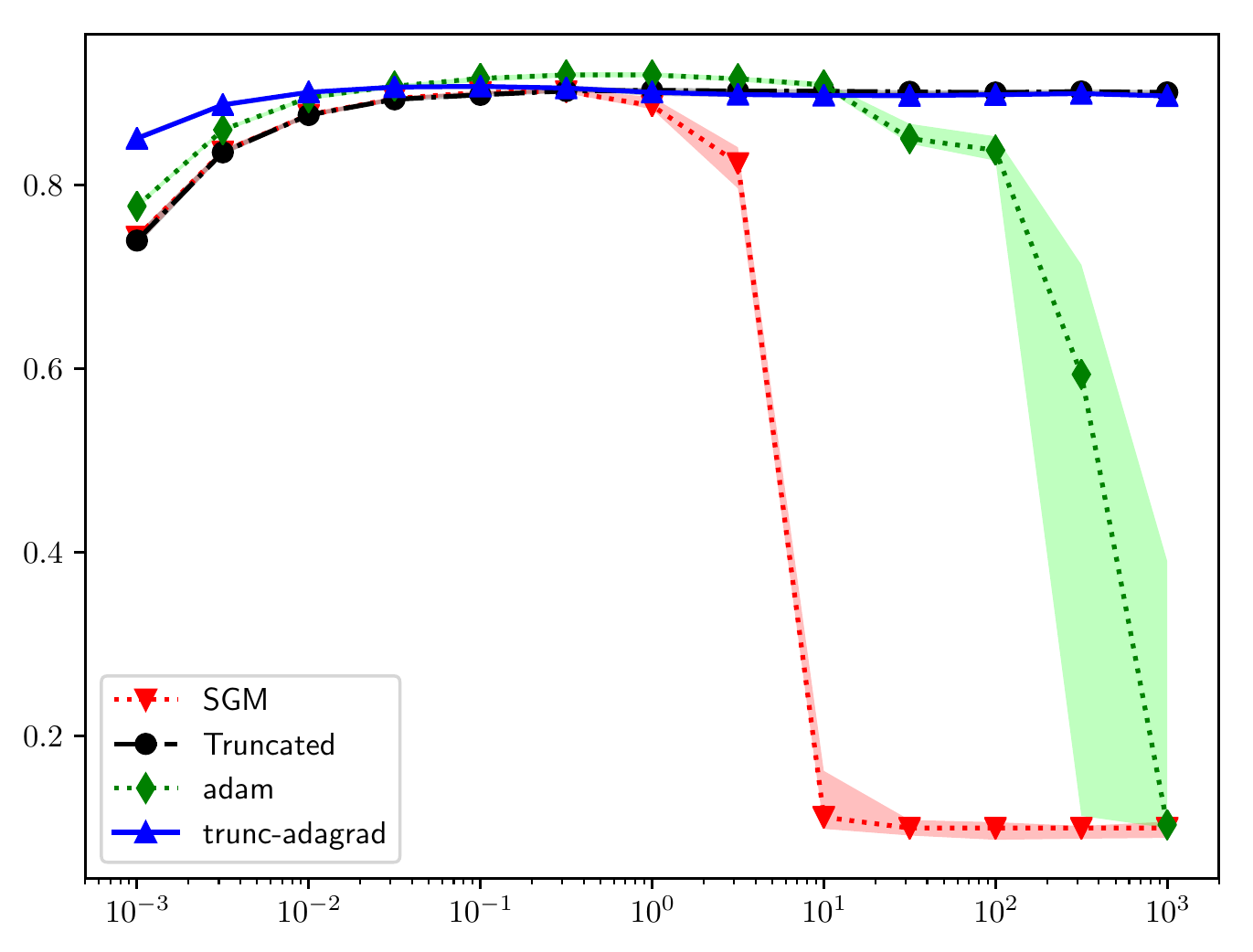}
	\put(35,-1){{\small Initial stepsize $\stepsize_0$}}
	\put(-5,30){
	  \rotatebox{90}{{\small Accuracy}}}
	%\put(33,72){
	%	\tikz{\path[draw=white,fill=white] (0, 0) rectangle (3cm,.35cm);}
	%}
      \end{overpic} \\
      (a) & (b)
    \end{tabular}
    \caption{\label{fig:cifar10} (a) The number of iterations to achieve
      $\epsilon$ test error as a function of the initial step size
      $\stepsize_0$ for CIFAR10.  (b) The best achieved accuracy after
      $T=50$ epochs.  }
  \end{center}
\end{figure}

\begin{figure}[ht]
  \begin{center}
    \begin{tabular}{cc}
      \begin{overpic}[width=.5\columnwidth]{%,grid]{%
	  %../../Theory/paper-non-convex/plots/NN/CIFAR_minLoss.pdf}
	  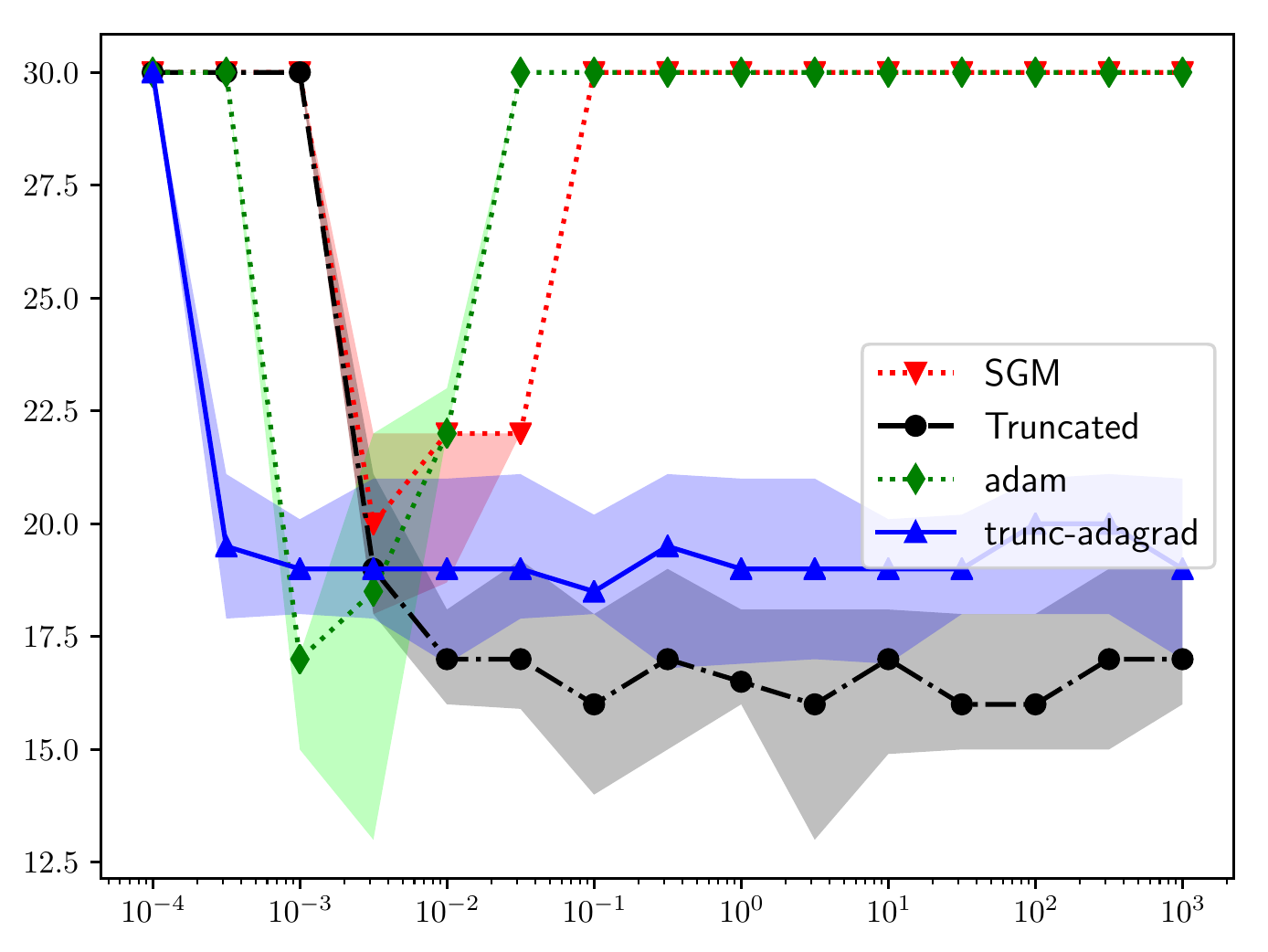}
	\put(35,-1){{\small Initial stepsize $\stepsize_0$}}
	\put(-4,15){
	  \rotatebox{90}{{\small Time to error $\epsilon = .35$}}}
	%\put(33,72){
	%	\tikz{\path[draw=white,fill=white] (0, 0) rectangle (3cm,.35cm);}
	%}
      \end{overpic} &
      \begin{overpic}[width=.5\columnwidth]{%,grid]{%
	  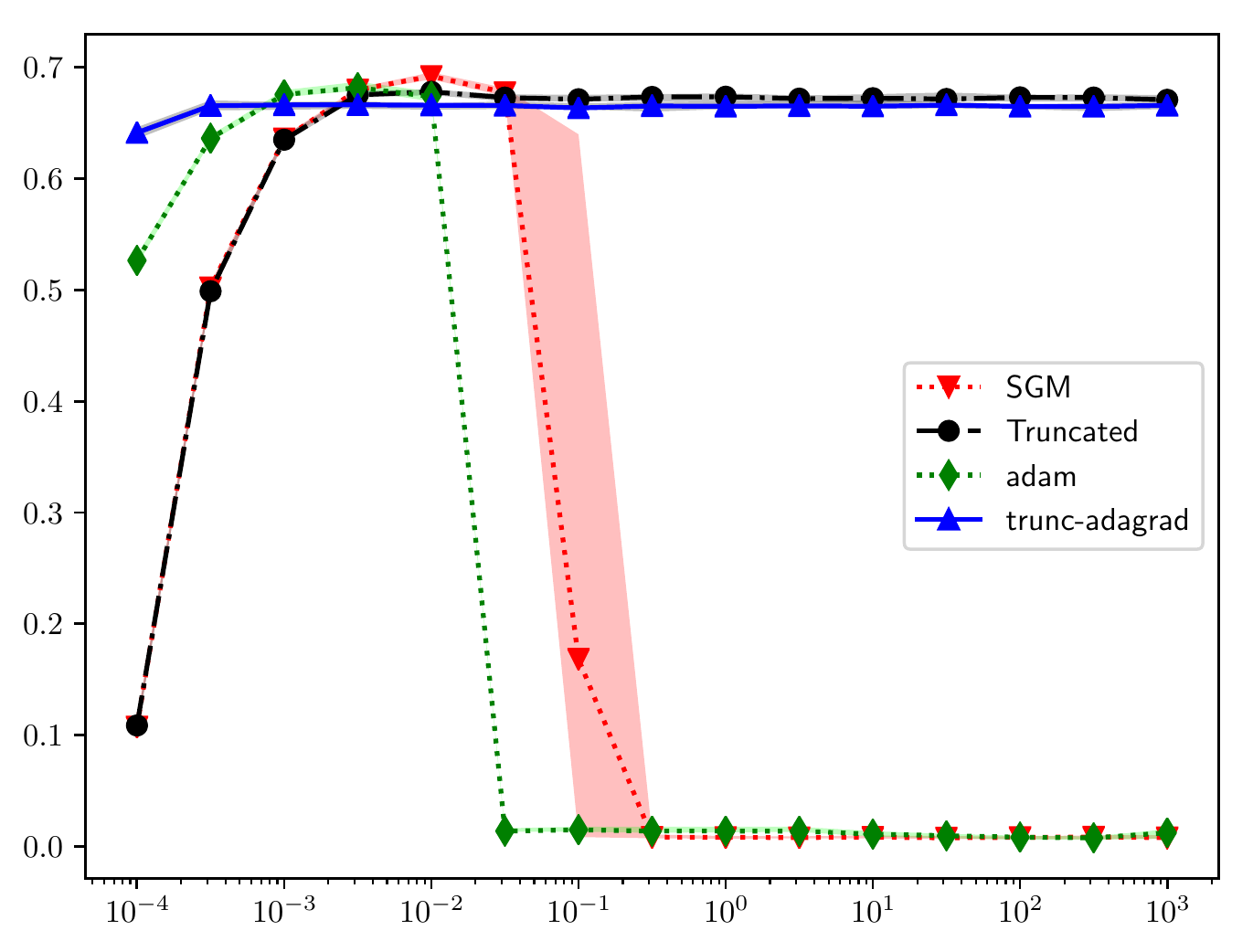}
	\put(35,-1){{\small Initial stepsize $\stepsize_0$}}
	\put(-3,30){
	  \rotatebox{90}{{\small Accuracy}}}
	%\put(33,72){
	%	\tikz{\path[draw=white,fill=white] (0, 0) rectangle (3cm,.35cm);}
	%}
      \end{overpic} \\
      (a) & (b)
    \end{tabular}
    \caption{\label{fig:stanford-dogs} (a) The number of iterations to
      achieve $\epsilon$ test error as a function of the initial step size
      $\stepsize_0$ for Stanford dogs dataset.  (b) The best achieved
      accuracy after $T=30$ epochs.  }
  \end{center}
\end{figure}

As one of our main motivations is to address the extraordinary effort---in
computational and engineering hours---carefully tuning optimization methods,
we would be remiss to avoid experiments on deep neural networks.  Therefore,
in our last set of experiments, we test the performance of our models for
training neural networks for classification tasks over the CIFAR10
dataset~\cite{KrizhevskyHi09} and the fine-grained 128-class Stanford dog
multiclass recognition task~\cite{KhoslaJaYaLi11}.  For our CIFAR10
experiment, we use the Resnet18 architecture~\cite{HeZhReSu16}; we
replace the Rectified Linear Unit (RELU) activations internal to the
architecture with Exponentiated Linear Units (ELUs~\cite{ClevertUnHo16}) so
that the loss is of composite form $f = h \circ c$ for $h$ convex and $c$
smooth.  For Stanford dogs we use the VGG16 architecture~\cite{SimonyanZi14}
pretrained on Imagenet~\cite{DengDoSoLiLiFe09}, again substituting ELUs for
RELU activations.  For this experiment, we also test a modified version of
the truncated method, \textsc{TruncAdaGrad}, that uses the truncated
model in iteration~\eqref{eqn:model-iteration} but uses a diagonally
scaled Euclidean
distance~\cite{DuchiHaSi11}, updating
at iteration $k$ by
\begin{equation*}
  x_{k + 1} = \argmin_x \left\{\hinge{f(x_k; \statrv_k) + \<g_k, x - x_k\>}
  + \frac{1}{2 \stepsize_0} (x - x_k)^T H_k (x - x_k)\right\},
\end{equation*}
where $H_k = \diag(\sum_{i = 1}^k g_i g_i^T)^{1/2}$ for $g_i = f'(x_i;
\statrv_i)$. This update requires no more out of standard deep learning
software than computing a gradient (backpropagation) and loss. We also
compare our optimization methods to \textsc{Adam}, the default optimizer in
TensorFlow~\cite{TensorFlow15}.

Figures~\ref{fig:cifar10} and~\ref{fig:stanford-dogs} show our results for
the CIFAR10 and Stanford dogs datasets, respectively. Plot~(a) of each
figure gives the number of iterations required to achieve $\epsilon$
test-classification error (on the highest or ``top-1'' predicted class),
while plot~(b) shows the maximal accuracy each procedure achieves for a
given initial stepsize $\stepsize_0$. The plots demonstrate the sensitivity
of the standard stochastic gradient method to stepsize choice, which
converges only for a small range of stepsizes, in both experiments. Adam
exhibits better robustness for CIFAR10, while it is extremely sensitive in
the second experiment (Fig.~\ref{fig:stanford-dogs}), converging only for a
small range of stepsizes---this difference in sensitivities highlights the
importance of robustness.  In contrast, our
procedures using the truncated model are apparently robust for all large
enough stepsizes. Plot~(b) of the figures shows additionally that the
maximal accuracy the two truncated methods achieve changes only slightly for
$\stepsize_0 \ge 10^{-1}$, again in strong contrast to the other methods,
which achieve their best accuracy only for a small range of stepsizes.

These results reaffirm the insights from our theoretical results and
experiments: it is important and possible to develop methods that enjoy good
convergence guarantees and are robust to algorithm parameters.

\appendix

% -*- Mode: latex -*- %

\section{Technical lemmas}

Before beginning our technical appendices, we collect a few results that
will be useful for our derivations to come. The first three provide
a guarantee on the progress and convergence that single iterates of
the updates~\eqref{eqn:model-iteration} make.

\begin{lemma}[Asi and Duchi~\cite{AsiDu18}, Lemma 3.2]
  \label{lemma:conv-bound}
  Let $g$ be convex and subdifferentiable on a closed convex set $\mc{X}$
  and let $\beta > 0$. Then for all $x_0, x_1, y \in \mc{X}$,
  and $g'(y) \in \partial g(y)$,
  \begin{equation*}
    g(y) - g(x_1)
    \le \<g'(y), y - x_0\>
    + \frac{1}{2 \beta} \norm{x_1 - x_0}^2
    + \frac{\beta}{2} \norm{g'(y)}^2
  \end{equation*}
\end{lemma}

\begin{lemma}[Asi and Duchi~\cite{AsiDu18}, Lemma 3.3]
  \label{lemma:single-step-progress}
  Let Condition~\ref{cond:convex-model} hold.
  In each step of the method~\eqref{eqn:model-iteration},
  for any $x \in \mc{X}$,
  \begin{equation*}
    \half \ltwo{x_{k + 1} - x}^2
    \le \half \ltwo{x_k - x}^2
    - \stepsize_k\left[f_{x_k}(x_{k + 1}; \statrv_k)
      - f_{x_k}(x; \statrv_k)\right]
    - \half \ltwo{x_k - x_{k+1}}^2.
  \end{equation*}
\end{lemma}

\begin{lemma}
  \label{lemma:single-recentered-progress}
  Let Conditions~\ref{cond:convex-model}--\ref{cond:subgrad-model} hold and
  let $x_k$ be generated by the updates~\eqref{eqn:model-iteration}. Then
  for any $x \in \mc{X}$,
  \begin{equation*}
    \half \ltwo{x_{k+1} - x}^2
    \le \half \ltwo{x_k - x}^2
    - \stepsize_k \left[f(x_k; \statrv_k) - f(x; \statrv_k)
      - \frac{\weakconvexfunc(\statrv_k)}{2} \ltwo{x_k - x}^2\right]
    + \frac{\stepsize_k^2}{2} \ltwo{f'(x_k; \statrv_k)}^2.
  \end{equation*}
\end{lemma}
\begin{proof}
  Let $\stepsize > 0$ and $x_0, x_1, x \in \mc{X}$ and
  $\statval \in \statdomain$ be otherwise arbitrary. Then
  for any $g \in \partial f(x_0; \statval)$,
  we have
  \begin{align*}
    f_{x_0}(x_1; \statval) - f_{x_0}(x; \statval)
    & \stackrel{(i)}{\ge} f_{x_0}(x_0; \statval) - f_{x_0}(x; \statval)
    + \<g, x_1 - x_0\> \\
    & \stackrel{(ii)}{\ge} f(x_0; \statval) - f(x; \statval)
    - \frac{\weakconvexfunc(\statval)}{2} \ltwo{x - x_0}^2
    - \frac{\stepsize}{2} \ltwo{g}^2
    - \frac{1}{2 \stepsize} \ltwo{x_0 - x_1}^2,
  \end{align*}
  where inequality~$(i)$ used the convexity of $f_{x_0}$ and
  inequality~$(ii)$ used Condtition~\ref{cond:subgrad-model} to obtain
  $f_{x_0}(x_0; \statval) = f(x_0; \statval)$ and
  Condition~\ref{cond:weak-convex-lower} on
  $f_{x_0}(x; \statval)$.
  Making the obvious substitutions for $x_0 \mapsto x_k$ and others in
  Lemma~\ref{lemma:single-step-progress} then gives the result.
  %% Using Lemma~\ref{lemma:single-step-progress}, it suffices to
  %% show that for any $\stepsize > 0$ and $x_0, x_1, x \in \mc{X}$
  %% \begin{equation*}
  %%   -\stepsize [f_{x_0}(x_1; \statval) - f_{x_0}(x; \statval)]
  %%   - \half \ltwo{x_1 - x_0}^2
  %%   \le -\stepsize [f(x_0; \statval) - f(x; \statval)]
  %%   + \frac{\stepsize^2}{2} \ltwo{f'(x_0; \statval)}^2.
  %% \end{equation*}
  %% To see this, note that
  %% \begin{align*}
  %%   -f_{x_0}(x_1; \statval) + f_{x_0}(x; \statval)
  %%   & = -[f_{x_0}(x_0; \statval) - f_{x_0}(x; \statval)]
  %%   + f_{x_0}(x_0; \statval) - f_{x_0}(x_1; \statval) \\
  %%   & \stackrel{\ref{cond:subgrad-model}}{\le}
  %%   -[f_{x_0}(x_0; \statval) - f_{x_0}(x; \statval)]
  %%   + \<f'(x_0;\statval), x_0 - x_1\> \\
  %%   & \stackrel{\ref{cond:lower-model}}{\le}
  %%   -[f(x_0;\statval) - f(x; \statval)]
  %%   + \<f'(x_0;\statval), x_0 - x_1\>.
  %% \end{align*}
  %% Then we use that for any vector $v$,
  %% $\stepsize \<v, \Delta\> - \half \ltwo{\Delta}^2 \le
  %% \frac{\stepsize^2}{2} \ltwo{v}^2$, which
  %% gives the result.
\end{proof}

We also need the Robbins-Siegmund almost supermartingale convergence theorem.
\begin{lemma}[\cite{RobbinsSi71}]
  \label{lemma:robbins-siegmund}
  Let $A_k, B_k, C_k, D_k \ge 0$ be non-negative random variables adapted to
  the filtration $\mc{F}_k$ and satisfying $\E[A_{k + 1} \mid \mc{F}_k] \le
  (1 + B_k) A_k + C_k - D_k$. Then on the event $\{\sum_k B_k < \infty,
  \sum_k C_k < \infty\}$, there is a random $A_\infty < \infty$
  such that $A_k \cas A_\infty$ and $\sum_k D_k < \infty$.
\end{lemma}

% -*- Mode: latex -*- %

\section{Proof of Theorem~\ref{theorem:trunc-stability}}
\label{sec:proof-truncated-stability}

We begin with two lemmas useful for the proof. The first lemma provides a
recursive inequality on the iterates for any model satisfying
Conditions~\ref{cond:convex-model}--\ref{cond:lower-by-optimal}.
\begin{lemma}
  \label{lemma:trunc-recurssion}
  Let $x_k$ be generated by the iteration~\eqref{eqn:model-iteration} with a
  model satisfying
  Conditions~\ref{cond:convex-model}--\ref{cond:lower-by-optimal}.  Then for
  any $x\opt \in \mc{X\opt}$
  \begin{equation*}
    \ltwo{x_{k + 1} - x\opt}^2
    \le  \ltwo{x_k - x\opt}^2
    + 2\stepsize_k
    \left( f(x\opt; \statrv_k) - \inf_{z\in \mc{X}} f(z;\statrv_k) \right) .
  \end{equation*}
\end{lemma}
\begin{proof}
  Lemma~\ref{lemma:single-step-progress} implies that
  \begin{align*}
    \half \ltwo{x_{k + 1} - x\opt}^2
    & \le  \half \ltwo{x_k - x\opt}^2
    + \stepsize_k \left[f_{x_k}(x\opt; \statrv_k) - f_{x_k}(x_{k+1}; \statrv_k)
      \right]
    - \half \ltwo{x_k - x_{k+1}}^2 \\
    & \stackrel{(\star)}{\le} \half \ltwo{x_k - x\opt}^2
    + \stepsize_k \left( f(x\opt; \statrv_k) 
    - \inf_{z\in \mc{X}} f(z;\statrv_k) \right),
  \end{align*}
  where inequality~$(\star)$ uses that $f_{x_k}(x\opt; \statrv_k)
  \le f(x\opt; \statrv_k)$ and $f_{x_k}(x_{k+1}; \statrv_k)
  \ge \inf_{z \in \mc{X}} f(z; \statrv_k)$.
\end{proof}
\noindent
The preceding recursion allows us to bound
the expected norm of the gradients of the iterates.
\begin{lemma}
  \label{lemma:bounded-expectation}
  Let Assumption~\ref{assumption:polynomial-growth} hold and let $x_k$ be
  generated by the iteration~\eqref{eqn:model-iteration} with a model
  satisfying
  Conditions~\ref{cond:convex-model}--\ref{cond:lower-by-optimal}
  and stepsizes $\stepsize_k = \stepsize_0 k^{-\steppow}$. Then
  there exist $C_1, C_2 < \infty$ such that
  \begin{equation*}
    \E\left[\ltwo{f'(x_k; \statrv)}^2\right]
    \le C_1 + C_2 k^{\frac{p (1-\steppow)}{2}}.
  \end{equation*}
\end{lemma}
\begin{proof}
  Assumption~\ref{assumption:polynomial-growth} implies that
  \begin{align*}
    \E\left[\ltwo{f'(x_k; \statrv)}^2\right] 
    & = \E\left[ \E\left[\ltwo{f'(x_k; \statrv)}^2 \mid \mc{F}_{k-1}
        \right] \right]
    \le  C(1+\E[ \dist(x_k, \mc{X}\opt)^p] ). 
  \end{align*}
  Let $f\opt_k = f(x\opt; \statrv_k) - \inf_{z\in \mc{X}} f(z;\statrv_k)$
  for shorthand.  Lemma~\ref{lemma:trunc-recurssion} implies that for any
  $x\opt \in \mc{X\opt}$,
  \begin{align*}
    \E\left[ \ltwo{x_k-x\opt}^p \right] 
    & \le \E\left[  \Big(\ltwo{x_1-x\opt}^2 
      + 2\sum_{i=1}^{k}{\stepsize_i f\opt_i  }\Big)^{p/2} \right] \\
    & \stackrel{(i)} 
    \le 2^{p/2}\ltwo{x_1-x\opt}^p 
    + 2^{p}\E\left[\bigg(\sum_{i=1}^{k}{\stepsize_i f\opt_i  }\bigg)^{p/2}
      \right]  \\
    %% & = 2^{p/2}\ltwo{x_1-x\opt}^p 
    %% + 2^{p}\E\left[\left(\frac{\sum_{i=1}^{k}{\stepsize_i f\opt_i}}{\sum_{i=1}^{k}{\stepsize_i}}\right)^{p/2}\right]  
    %% \left(\sum_{i=1}^{k}{\stepsize_i}\right)^{p/2} \\
    & \stackrel{(ii)} 
    \le 2^{p/2}\ltwo{x_1-x\opt}^p 
    + 2^{p} \frac{\sum_{i=1}^{k}{\stepsize_i \E[(f\opt_i)^{p/2}]}}{
      \sum_{i=1}^{k} \stepsize_i}
    \bigg(\sum_{i=1}^{k}{\stepsize_i}\bigg)^{p/2} \\
    & \stackrel{(iii)}
    \le 2^{p/2}\ltwo{x_1-x\opt}^p 
    + C k^{\frac{p (1-\beta)}{2}}
  \end{align*}
  where inequality~$(i)$ is a consequence of the inequality $(a+b)^\ell \le
  2^\ell (a^\ell+b^\ell)$, which holds for any $a,b,\ell \ge 0$,
  inequality~$(ii)$ follows from Jensen's inequality, and inequality~$(iii)$
  follows from Assumption~\ref{assumption:polynomial-growth} and
  that $\stepsize_i \propto i^{-\steppow}$ so that
  $\sum_{i=1}^k \stepsize_i \le C k^{1 - \steppow}$.
\end{proof}

We use the preceding lemmas to prove Theorem~\ref{theorem:trunc-stability}.
Lemma~\ref{lemma:single-recentered-progress}, with the weak convexity
constant $\weakconvexfunc \equiv 0$, implies that
if $x\opt$ is the projection of $x_k$ onto $\mc{X}\opt$, then
\begin{equation*}
  \E[\ltwo{x_{k + 1} - x\opt}^2 \mid \mc{F}_{k-1}]
  \le  \dist(x_k, \mc{X}\opt)^2
  - 2\stepsize_k [F(x_k) - F(x\opt)]
  + \stepsize_k^2
  \E \left[\ltwo{{f'(x_k;\statrv_k)}}^2 \mid \mc{F}_{k-1} \right].
\end{equation*}
Let us denote $B_k = \stepsize_k^2 \E [\ltwo{{f'(x_k;\statrv_k)}}^2 \mid
  \mc{F}_{k-1}]$.  By Lemma~\ref{lemma:robbins-siegmund}, Theorem~\ref{theorem:trunc-stability}---that the $x_k$ are bounded---will follow if we can
prove that $\sum_{k=1}^{\infty} B_k < \infty$ almost surely.
As $\E[B_k] = \stepsize_k^2 \E[\ltwo{f'(x_k; \statrv_k)}^2] $,
Lemma~\ref{lemma:bounded-expectation} implies that for some $C_1,
C_2 < \infty$,
\begin{equation*}
  \E\left[ \sum_{i=1}^{k} B_i \right]
  \le C_1 \sum_{i=1}^{k}{k^{-2\beta}}
  + C_2 \sum_{i=1}^{k}{k^{-\gamma}},
\end{equation*}
where $\gamma = -p/2 +\beta(p/2+2)$. Whenever $\frac{p+2}{p+4} < \steppow <
1$, we get that $\gamma > 1$, implying that $\lim_{k\to\infty} \E[
  \sum_{i=1}^{k} B_i ] < \infty$. The monotone convergence theorem implies
that $\E[ \sum_{k=1}^{\infty} B_k ] < \infty$.  We conclude that
$\sum_{k=1}^{\infty} B_k < \infty$ almost surely, completing the proof.

% -*- Mode: latex -*- %

\section{Proof of Theorem~\ref{theorem:convergence-basic}}
\label{sec:proof-convergence-basic}

We assume without loss of generality that $\inf_{x \in \mc{X}} F(x) = 0$. It
is enough to prove that the recursion in the next lemma holds, because
Eq.~\eqref{eqn:moreau-useful} implies that $\ltwo{x^\lambda_k - x_k} =
\ltwo{\nabla F_\lambda(x_k)} / \lambda$. The theorem will then follow from
the Robbins-Siegmund almost supermartingale convergence lemma
(Lemma~\ref{lemma:robbins-siegmund}).

\begin{lemma}
  \label{lemma:two-weak-moments}
  %Let Assumption~\ref{assumption:variance-bounded} hold with
  %constants $C_1, C_2 < \infty$ and let $x_k$ be generated 
  %by the iteration~\eqref{eqn:model-iteration} 
  %with the model~\eqref{eqn:prox-model}. Assume that
  %$\E[\weakconvexfunc(\statrv_k)] < \lambda$
  %and $\E[\weakconvexfunc(\statrv)^2] < \infty$.
  Let the conditions of Theorem~\ref{theorem:convergence-basic}
  hold and let $A_\lambda = 2 (C_1 \lambda^2
  + \E[(\weakconvexfunc(\statrv) - \lambda)^2])$. Then
  \begin{equation*}
    \E[ F_\lambda(x_{k+1})  \mid \mc{F}_{k-1}]
    \le (1 + A_\lambda \stepsize_k^2) F_\lambda(x_k)
    - \lambda \stepsize_k \left(\lambda - \E[\weakconvexfunc(\statrv)]\right)
    \ltwobig{x^\lambda_k-x_k}^2
    + \lambda  \stepsize_k^2 C_2.
  \end{equation*}
\end{lemma}
\begin{proof}
	First, we have 
	\begin{align}
	\label{eqn:one-step-moreau}
	F_\lambda(x_{k+1})
	\le F(x^\lambda_{k}) + \frac{\lambda}{2} \ltwos{x^\lambda_{k}
		- x_{k+1}}^2.
	\end{align}
	Now we bound the term $\ltwos{x^\lambda_{k}- x_{k+1}}^2$. Since
	$x_{k+1}$ solves the update~\eqref{eqn:model-iteration}
	with the model $f_{x_k}$, we have by the standard optimality
	conditions for strongly convex minimization that
	\begin{align*}
	\half \ltwo{x_{k+1} - y}^2
	\le \half \ltwo{x_k - y}^2 
	+ \stepsize_k (f_{x_k}(y; \statrv_k) 
	-  f_{x_k}(x_{k+1};\statrv_k))
	- \half \ltwo{x_{k+1} - x_k}^2 .
	\end{align*} 
	Now, we take $y = x^\lambda_k$ in the last inequality and apply
	Lemma~\ref{lemma:conv-bound} with $x_1 = x_{k+1}$, $x_0 = x_k$, $y =
	x^\lambda_k$, and $\beta = \stepsize_k$ to find
	\begin{equation}
	\label{eqn:half-a-moreau-recursion}
	\half \ltwo{x_{k + 1} - x^\lambda_k}^2
	\le \half \ltwo{x_k - x^\lambda_k}^2
	+ \stepsize_k \<f'_{x_k}(x^\lambda_k; \statrv_k), x^\lambda_k - x_k\>
	+ \frac{\stepsize_k^2}{2}
	\ltwo{f'_{x_k}(x^\lambda_k; \statrv_k)}^2
	\end{equation}
	for all $f'(x^\lambda_k; \statrv_k) \in \partial f(x^\lambda_k;
	\statrv_k)$. But since $f_{x_k}(x; \statrv_k) = f(x; \statrv_k) +
	\frac{\weakconvexfunc(\statrv_k)}{2} \ltwo{x-x_k}^2$, we have that
	\begin{align*}
	\<f'_{x_k}(x^\lambda_k; \statrv_k), x^\lambda_k - x_k\>
	& =  \<f'(x^\lambda_k; \statrv_k) + \weakconvexfunc(\statrv_k) (x^\lambda_k-x_k), x^\lambda_k - x_k\> \\
	& = \<f'(x^\lambda_k; \statrv_k) + \lambda(x^\lambda_k-x_k),  x^\lambda_k - x_k\> 
	+ (\weakconvexfunc(\statrv_k) - \lambda) \ltwo{x^\lambda_k-x_k}^2  
	\end{align*}
	Taking expectations, the definition of $x_k^\lambda$ as the minimizer of
	$F(x) + \frac{\lambda}{2} \ltwo{x - x_k}^2$ over $\mc{X}$ implies that
	w.l.o.g.\ we have $\E[f'(x_k^\lambda; \statrv_k) \mid \mc{F}_{k-1}] =
	F'(x_k^\lambda)$ for the element $F'(x_k^\lambda) \in \partial
	F(x_k^\lambda)$ satisfying $\<F'(x_k^\lambda) + \lambda(x_k^\lambda -
	x_k), y - x_k^\lambda\> \ge 0$ for all $y \in \mc{X}$. Thus we obtain that
	\begin{equation}
	\label{eqn:eq1}
	\E[\<f'_{x_k}(x^\lambda_k; \statrv_k), x^\lambda_k - x_k\> \mid \mc{F}_{k-1}]
	\le \E[\weakconvexfunc(\statrv) - \lambda ] \ltwobig{x^\lambda_k-x_k}^2.
	\end{equation}
	Moreover, we have
	\begin{align}
	\ltwo{f'_{x_k}(x^\lambda_k; \statrv_k)}^2
	& = \ltwo{f'(x^\lambda_k; \statrv_k) + \weakconvexfunc(\statrv_k) (x^\lambda_k-x_k)}^2 \nonumber  \\
	& = \ltwo{f'(x^\lambda_k; \statrv_k) + \lambda(x^\lambda_k-x_k) + (\weakconvexfunc(\statrv_k) - \lambda) (x^\lambda_k-x_k)}^2 \nonumber  \\
	& \le 2 \ltwo{f'(x^\lambda_k; \statrv_k) + \lambda(x^\lambda_k-x_k)}^2
	+ 2 (\weakconvexfunc(\statrv_k) - \lambda)^2 \ltwo{x^\lambda_k-x_k}^2.
	\label{eqn:controlling-f-prime-fun}
	\end{align}
	Combining Assumption~\ref{assumption:variance-bounded}
	with the fact that $\E[f'(x_k^\lambda; \statrv_k) \mid \mc{F}_{k-1}]
	= F'(x_k^\lambda) = -\lambda(x_k^\lambda - x_k)$,
	we have that
	\begin{equation*}
	\E[\ltwos{f'(x^\lambda_k; \statrv_k)
		+ \lambda (x^\lambda_k - x_k)}^2 \mid \mc{F}_{k-1}]
	% = \var(f'(x^\lambda_k; \statrv_k) \mid \mc{F}_{k-1})
	\le C_1 \ltwos{F'(x_k^\lambda)}^2 + C_2
	= C_1 \lambda^2 \ltwobig{x_k^\lambda - x_k}^2 + C_2.
	\end{equation*}
	Substituting this bound into inequality~\eqref{eqn:controlling-f-prime-fun}
	gives
	\begin{equation}
	\E\left[\ltwobig{f_{x_k}'(x^\lambda_k; \statrv_k)}^2 \mid \mc{F}_{k-1}\right]
	\le 2 \left(C_1 \lambda^2 + \E[(\weakconvexfunc(\statrv_k)
	- \lambda)^2 \mid \mc{F}_{k-1}]\right)
	\ltwobig{x_k^\lambda - x_k}^2
	+ 2 C_2.
	\label{eqn:eq2}
	\end{equation}
	
	Now we derive our final recursion.  Taking expectations in the
	bound~\eqref{eqn:half-a-moreau-recursion} and applying
	inequalities~\eqref{eqn:eq1} and~\eqref{eqn:eq2}
	with the choice $A = 2 (C_1 \lambda^2 +
	\E[(\weakconvexfunc(\statrv) - \lambda)^2])$, we get
	\begin{equation*}
	\E\left[\ltwo{x_{k + 1} - x^\lambda_k}^2 \mid \mc{F}_{k-1}\right]
	\le  \left( 1 + \stepsize_k^2 A \right) \ltwobig{x_k - x^\lambda_k}^2
	+ 2 \stepsize_k
	\left[\wb{\weakconvexfunc} - \lambda \right]\ltwobig{x^\lambda_k-x_k}^2
	+ 2 C_2 \stepsize_k^2.
	\end{equation*}
	Using the convexity bound~\eqref{eqn:one-step-moreau}, we
	thus have
	\begin{align*}
	\E\left[F_\lambda(x_{k+1}) \mid \mc{F}_{k-1}\right]
	& \le F(x^\lambda_{k}) + \frac{\lambda}{2}
	\E\left[\ltwobig{x^\lambda_{k} - x_{k+1}}^2 \mid \mc{F}_{k-1}\right] \\
	& \le F(x^\lambda_{k}) + \frac{\lambda}{2}
	\left( 1 + \stepsize_k^2 A \right) \ltwobig{x_k - x^\lambda_k}^2
	+ \lambda \stepsize_k \left[\wb{\weakconvexfunc} - \lambda\right]
	\ltwobig{x^\lambda_k-x_k}^2
	+ \lambda C_2 \stepsize_k^2 \\
	& \le \left( 1 + \stepsize_k^2 A \right) \left( F(x^\lambda_{k})
	+ \frac{\lambda}{2} \ltwobig{x_k - x^\lambda_k}^2 \right)
	+ \lambda \stepsize_k
	\left[\wb{\weakconvexfunc} - \lambda \right] \ltwobig{x^\lambda_k-x_k}^2
	+ \lambda C_2 \stepsize_k^2.
	%% & \le \left( 1 + \stepsize_k^2 A \right) F_\lambda(x_{k})
	%% + 2 \stepsize_k \E[ \weakconvexfunc(\statrv_k) - \lambda ] \ltwo{x^\lambda_k-x_k}^2 
	%% + \lambda C_2 \stepsize_k^2 
	\end{align*}
	Using that $F_\lambda(x) = F(x) + \frac{\lambda}{2} \ltwos{x^\lambda - x}^2$
	gives the lemma.
\end{proof}

\section{Proof of
  Proposition~\ref{proposition:convergence-from-boundedness-weakly-convex}}

Throughout the proof, let
$\wb{\weakconvexfunc}=\E[\weakconvexfunc(\statrv)]$ be the expected weak
convexity constant.

We prove the first claim of the proposition.
By Lemma~\ref{lemma:single-recentered-progress}, we have
\begin{align*}
  \lefteqn{F_\lambda(x_{k+1})
    \le F(x^{\lambda}_k) + \frac{\lambda}{2}
    \ltwos{x^\lambda_k - x_{k+1}}^2} \\
  & \le F(x^\lambda_k) + \frac{\lambda}{2}
  \ltwos{x_k^\lambda - x_k}^2
  - \lambda \stepsize_k \left[f(x_k; \statrv_k) - f(x^\lambda_k; \statrv_k)
    - \frac{\weakconvexfunc(\statrv_k)}{2} \ltwos{x_k - x_k^\lambda}^2\right]
  + \frac{\lambda \stepsize_k^2}{2} \ltwos{f'(x_k; \statrv_k)}^2 \\
  & = F_\lambda(x_k) - \lambda
  \stepsize_k\left[f(x_k; \statrv_k) - f(x^\lambda_k; \statrv_k)
    - \frac{\weakconvexfunc(\statrv_k)}{2} \ltwos{x_k - x_k^\lambda}^2\right]
  + \frac{\lambda \stepsize_k^2}{2} \ltwos{f'(x_k; \statrv_k)}^2.
\end{align*}
Now, taking expectations conditional on $\mc{F}_{k-1}$, we have
\begin{align*}
  \E[F_\lambda(x_{k+1}) \mid \mc{F}_{k-1}]
  & \le F_\lambda(x_k^\lambda)
  - \lambda \stepsize_k \left[F(x_k) - F(x_k^\lambda)
    - \frac{\wb{\weakconvexfunc}}{2} \ltwos{x_k - x_k^\lambda}^2\right]
  + \frac{\lambda \stepsize_k^2}{2} \growfunc(\ltwo{x_k}),
\end{align*}
where we have used Assumption~\ref{assumption:weak-second-moment}
and the definition of $\growfunc$.
Adding and subtracting
$\frac{\lambda \stepsize_k}{2} \ltwos{x_k - x_k^\lambda}^2
= \frac{\stepsize_k}{2 \lambda} \ltwos{\nabla F_\lambda(x_k)}^2$, we have
\begin{align*}
  \lefteqn{\E[F_\lambda(x_{k+1}) \mid \mc{F}_{k-1}]} \\
  & \le F_\lambda(x_k^\lambda)
  - \lambda \stepsize_k \left[F(x_k) - F_\lambda(x_k)\right]
  + \frac{(\wb{\weakconvexfunc} - \lambda) \stepsize_k}{2 \lambda}
  \ltwo{\nabla F_\lambda(x_k)}^2
  + \frac{\lambda \stepsize_k^2}{2} \growfunc(\ltwo{x_k}).
\end{align*}
As $\sup_k \ltwo{x_k} < \infty$ by assumption,
the final term above satisfies
$\growfunc(\ltwo{x_k}) \le B < \infty$ for some (random) $B$ for all $k$.
The
martingale convergence Lemma~\ref{lemma:robbins-siegmund} then
gives that $\sum_k \stepsize_k \ltwo{\nabla F_\lambda(x_k)}^2 < \infty$
with probability 1, establishing the convergence
guarantees~\eqref{eqn:boundedness-to-stationary-convergence}.

Now for the second claim of
Proposition~\ref{proposition:convergence-from-boundedness-weakly-convex},
which relies on Assumption~\ref{assumption:weak-sard}.
We
collect a few useful results. First, we note that
$F_\lambda(\mc{X}\stationary) = F(\mc{X}\stationary)$ for all $\lambda \ge
\wb{\weakconvexfunc}$, as any $x \in \mc{X}\stationary$ is a minimizer of $y
\mapsto F(y) + \frac{\lambda}{2} \ltwo{y - x}^2$.
Second, we claim that
\begin{equation}
  \label{eqn:convergence-to-stationary-values}
  \dist(F_\lambda(x_k), F(\mc{X}\stationary)) \cas 0.
\end{equation}
For the sake of contradiction, assume that this fails, noting that we have
$F_\lambda(x_k) \cas G_\lambda$ for some $G_\lambda < \infty$.  We know
that $x \mapsto \nabla F_\lambda(x)$ is a continuous function, and so over
compacta, $\norm{\nabla F_\lambda(x)}$ achieves its minimum.  Let $B <
\infty$ be such that $\norm{x_k} < B$ for all $k$ ($B$ may be random, but
it exists). Then $\dist(F_\lambda(x_k), F(\mc{X}\stationary)) \to \delta >
0$. As $y \mapsto \dist(F_\lambda(y), F(\mc{X}\stationary))$ is
continuous, the pre-image $Y_\delta \defeq \{y \in \mc{X} \mid \norm{y}
\le B, \dist(F_\lambda(y), F(\mc{X}\stationary)) \ge \delta/2\}$ is
compact.  Then $\norm{\nabla F_\lambda(y)}$ achieves its infimum over
$Y_\delta$, which must be positive (as otherwise there would exist a
stationary point in $Y_\delta$, a contradiction). But then for $K$ large
enough that $\dist(F_\lambda(x_k), F(\mc{X}\stationary)) > \delta/2$ for
$k \ge K$, we have
\begin{equation*}
  \sum_{k \ge K} \stepsize_k \ltwo{\nabla F_\lambda(x_k)}^2
  \ge \sum_{k \ge K} \stepsize_k \inf_{y \in Y_\delta} \ltwo{\nabla
    F_\lambda(y)}^2.
\end{equation*}
The first sum is finite by the first part of the proposition, a
contradiction to the fact that $\sum_k \stepsize_k = \infty$ and $\inf_{y
  \in Y_\delta} \ltwos{\nabla F_\lambda(y)}^2 > 0$.

Now, consider the limiting value $G(\lambda) = \lim_k F_\lambda(x_k)$.
Then for any limit point $x_\infty$ of $x_k$, we have $G(\lambda) =
F_\lambda(x_\infty)$, and thus $G(\lambda) \in F(\mc{X}\stationary)$.  As
by assumption $F(\mc{X}\stationary)$ has measure zero, it must be the case
that $G(\lambda)$ is constant in $\lambda$.  That is, $x_\infty$ satisfies
\begin{equation*}
  \inf_{x \in \mc{X}} \left\{F(x) + \frac{\lambda}{2} \ltwo{x - x_\infty}^2
  \right\} = G
\end{equation*}
for some fixed $G$. We now make the following claim.
\begin{lemma}
  \label{lemma:constant-envelopes-means-min}
  Let $h$ be convex and assume that $h_\lambda(x_0)
  = \inf_x \{h(x) + \frac{\lambda}{2} \ltwo{x - x_0}^2\}$ is constant
  in $\lambda$. Then $x_0$ minimizes $h$.
\end{lemma}
\noindent
Deferring the proof of Lemma~\ref{lemma:constant-envelopes-means-min}
temporarily, note that $y \mapsto F(y) + \frac{\lambda}{2} \ltwo{y -
  x_\infty}^2$ is convex for large enough $\lambda$. But
of course $F_\lambda(x_\infty)$ is constant in $\lambda$, and thus
$x_\infty$ must minimize $F(x) + \frac{\lambda}{2} \ltwo{x - x_\infty}^2$,
so that $x_\infty$ is stationary.

Summarizing, we have shown that all subsequences of $x_k$ have a further
subsequence that converges to some $x_\infty \in \mc{X}\stationary$, and thus
$\dist(x_k, \mc{X}\stationary) \cas 0$.

\begin{proof-of-lemma}[\ref{lemma:constant-envelopes-means-min}]
  Let $h'(x_0; d) = \lim_{t \downarrow 0} \frac{h(x_0 + td) - h(x_0)}{t}$
  be the directional derivative of $h$, which exists and satisfies $h(x_0
  + td) - h(x_0) - t h'(x_0; d) = o(t)$ as $t \downarrow 0$
  (cf.~\cite[Lemma VI.2.2.1]{HiriartUrrutyLe93ab}). Moreover, $x_0$
  minimizes $h$ if and only if $h'(x_0; d) \ge 0$ for all $d \in \R^n$,
  and $\lim_{\lambda \uparrow \infty} h_\lambda(x) = h(x)$ for all $x$
  (cf.~\cite[Proposition XV.4.1.5]{HiriartUrrutyLe93ab}). Now, let $d$
  satisfy $\ltwo{d} = 1$, and assume for the sake of contradiction that
  $h'(x_0; d) < 0$. Then
  \begin{equation*}
    h(x_0 + td) + \frac{\lambda}{2} t^2 \ltwo{d}^2
    = h(x_0) + t h'(x_0; d) + \frac{\lambda}{2} t^2
    + o(t),
  \end{equation*}
  and taking $t = -\frac{1}{\lambda} h'(x_0; d)$ and letting $\lambda
  \uparrow \infty$, we have
  \begin{equation*}
    h_\lambda(x_0) \le
    h(x_0 + td) + \frac{\lambda}{2} t^2
    = h(x_0) - \frac{1}{2 \lambda} |h'(x_0; d)|^2 + o(\lambda^{-1})
    < h(x_0),
  \end{equation*}
  which contradicts that $h_\lambda(x_0)$ is constant in $\lambda$.
\end{proof-of-lemma}
%\end{proof}

% -*- Mode: latex -*- %

\section{Proofs of fast convergence on easy problems}	

\subsection{Proof of Lemma~\ref{lemma:shared-min-progress}}
We assume without loss of generality that $f(x\opt;\statval) = 0$ for
all $x\opt \in \mc{X}\opt$, as we may replace $f$ with $f - \inf f$.
Lemma 3.3 from~\cite{AsiDu18} implies
\begin{equation*}
  \half \ltwo{x_{k+1} - x\opt}^2
  \le \half \ltwo{x_k - x\opt}^2
  + \stepsize_k \left[f_{x_k}(x\opt;\statrv_k)
    - f_{x_k}(x_{k+1}; \statrv_k)\right] - \half \ltwo{x_{k+1} - x_k}^2.
\end{equation*}
For shorthand, let $g_k = f'(x_k;\statrv_k)$ and $f_k = f(x_k;\statrv_k)$.
Now, note that
$f_{x_k}(x\opt; \statrv_k) \le f(x\opt; \statrv_k) + \frac{\weakconvexfunc(\statrv_k)}{2}\ltwo{x_k - x\opt}^2 $,
and by Condition~\ref{cond:lower-by-optimal} we have
\begin{equation*}
  f_{x_k}(x_{k+1}; \statrv_k)
  \ge \left[f_{x_k}(x_k; \statrv_k)
    + \<g_k, x_{k+1} - x_k\>\right]
  \vee \inf_{x\opt \in \mc{X}} f(x\opt; \statrv_k)
  = \hinge{f_k + \<g_k, x_{k+1} - x_k\>}.
\end{equation*}
Thus we have
\begin{align}
  \label{eqn:single-step-with-hinge}
  \half & \ltwo{x_{k+1} - x\opt}^2 \\
  & \le \frac{1 + \stepsize_k \weakconvexfunc(\statrv_k)}{2}
  \ltwo{x_k - x\opt}^2
  + \stepsize_k \left[f(x\opt;\statrv_k)
    - \hinge{f_k + \<g_k, x_{k+1} - x_k\>} \right]
  - \half \ltwo{x_{k+1} - x_k}^2. \nonumber
\end{align}

If we let $\wt{x}_{k+1}$ denote the unconstrained minimizer
\begin{equation*}
  \wt{x}_{k+1} = \argmin\left\{\hinge{f_k + \<g_k, x - x_k\>}
  + \frac{1}{2\stepsize_k} \ltwo{x - x_k}^2\right\}
  = x_k - \lambda_k g_k
  ~~ \mbox{for}~~
  \lambda_k \defeq \min\left\{\stepsize_k, \frac{f_k}{
    \ltwo{g_k}^2}\right\},
\end{equation*}
then because $x_{k+1} \in \mc{X}$ we have
\begin{equation*}
  -\stepsize_k f_{x_k}(x_{k+1};\statrv_k)
  - \half \ltwo{x_{k+1} - x_k}^2
  \le -\stepsize_k f_{x_k}(\wt{x}_{k+1}; \statrv_k)
  - \half \ltwo{\wt{x}_{k+1} - x_k}^2.
\end{equation*}
By inspection, the guarded stepsize $\lambda_k$ guarantees that
$\hinge{f_k + \<g_k, \wt{x}_{k+1} - x_k\>} = f_k - \lambda_k
\ltwo{g_k}^2$, and thus inequality~\eqref{eqn:single-step-with-hinge}
(setting $f(x\opt;\statrv_k) = 0$) yields
\begin{align*}
  \half \ltwo{x_{k+1} - x\opt}^2
  & \le \frac{1 + \stepsize_k \weakconvexfunc(\statrv_k)}{2}
  \ltwo{x_k - x\opt}^2
  - \stepsize_k f_{x_k}(\wt{x}_{k+1}; \statrv_k)
  - \half \ltwos{\wt{x}_{k+1} - x_k}^2 \\
  & = \frac{1 + \stepsize_k \weakconvexfunc(\statrv_k)}{2}
  \ltwo{x_k - x\opt}^2
  - \lambda_k f_k
  + \frac{\lambda_k^2}{2} \ltwo{g_k}^2.
\end{align*}
We have two possible cases: whether
$f_k / \ltwo{g_k}^2 \lessgtr \stepsize_k$.
In the case that $f_k / \ltwo{g_k}^2 \le \stepsize_k$,
we have $\lambda_k = f_k / \ltwo{g_k}^2$ and so
$-\lambda_k f_k + \lambda_k^2 \ltwo{g_k}^2 / 2 =
- f_k^2 / (2 \ltwo{g_k}^2)$. In the alternative
case that $f_k / \ltwo{g_k}^2 > \stepsize_k$, we have
$\lambda_k = \stepsize_k$ and
$\stepsize_k^2 \ltwo{g_k}^2 / 2 \le
\stepsize_k f_k / 2$. Combining these cases
gives
\begin{equation*}
  \half \ltwo{x_{k+1} - x\opt}^2
  \le \frac{1 + \stepsize_k \weakconvexfunc(\statrv_k)}{2}
  \ltwo{x_k - x\opt}^2
  - \half \min\left\{\stepsize_k f_k,
  \frac{f_k^2}{\ltwo{g_k}^2}\right\},
\end{equation*}
which is the desired result.

\subsection{Proof of
  Proposition~\ref{proposition:sharp-growth-convex-convergence}}

Let $D_k = \dist(x_k, \mc{X}\opt)$ for shorthand throughout this proof.
Lemma~\ref{lemma:shared-min-progress} implies that
\begin{align*}
  \E[D_{k+1}^2 \mid \mc{F}_{k-1}]
  & \le (1 + \stepsize_k \wb{\weakconvexfunc}) D_k^2
  - \min\left\{\lambda_0 \stepsize_k D_k,
  \lambda_1 D_k^2 \right\}.
\end{align*}
Let $\epsilon > 0$ be arbitrary. Then
\begin{align}
  %% \lefteqn{(1 - \lambda)^{k+1} \E[V_{k,m} \mid \mc{F}_{k-1}]} \\
  \lefteqn{\E[D_{k+1}^2 \mid \mc{F}_{k-1}]
    \indic{\max\{D_m, D_{m+1}, \ldots, D_k\}
      \le \frac{\lambda_0}{(1 + \epsilon) \wb{\weakconvexfunc}}}}
  \nonumber \\
  & \le \max\{(1 + \stepsize_k(\wb{\weakconvexfunc} - \lambda_0 / D_k)),
  1 - \lambda_1 \} D_k^2
  \indic{\max\{D_m, D_{m+1}, \ldots, D_k\}
    \le  \frac{\lambda_0}{(1 + \epsilon) \wb{\weakconvexfunc}}}
  \nonumber \\
  & \le \max\{1 - \epsilon \wb{\weakconvexfunc}
  \stepsize_k, 1 - \lambda_1\}
  D_k^2 \indic{\max\{D_m, D_{m+1}, \ldots, D_{k-1}\}
    \le \frac{\lambda_0}{(1 + \epsilon) \wb{\weakconvexfunc}}}.
  \label{eqn:one-step-weak-contraction}
\end{align}
In particular, taking $K_0 \ge m$ to be the smallest integer such that
$k \ge K_0$ implies $\epsilon \wb{\weakconvexfunc} \stepsize_k
\le \lambda_1$,
\begin{equation*}
  K_0 =
  \floor{\left(\frac{\epsilon \wb{\weakconvexfunc} \stepsize_0}{\lambda_1}
    \right)^{1/\beta}} \vee m,
\end{equation*}
we obtain that
\begin{align*}
  \lefteqn{\E\left[D_{k+1}^2 \indic{
	\max\{D_m, \ldots, D_k\} \le (1 + \epsilon)^{-1} \lambda_0
	/ \wb{\weakconvexfunc}}\right]} \\
  & \le \E[D_m^2 \indic{D_m \le (1 + \epsilon)^{-1} \lambda_0
      / \wb{\weakconvexfunc}}]
  \exp\left(-\lambda_1 \hinge{\min\{k, K_0\} - m}
  - \epsilon \wb{\weakconvexfunc} \sum_{i = K_0 + 1}^k \stepsize_i \right) \\
  & \le \frac{\lambda_0^2}{(1 + \epsilon)^2 \wb{\weakconvexfunc}^2}
  \exp\left(-\lambda_1 \hinge{\min\{k, K_0\} - m}
  - \epsilon \wb{\weakconvexfunc} \sum_{i = K_0 + 1}^k \stepsize_i \right).
\end{align*}

Now, we give an argument analogous to our proof of linear convergence for
convex problems~\cite[Proposition~2]{AsiDu18}.
We have for any $\epsilon > 0$, $\delta > 0$ that
\begin{align*}
  \sum_{k = 1}^\infty
  \P\left(D_k \indic{\max_{m \le i < k}\{D_i\}
    \le
    \frac{\lambda_0}{(1 + \epsilon) \wb{\weakconvexfunc}}}
  \ge \delta \stepsize_k\right)
  & \le \frac{1}{\delta} \sum_{k = 1}^\infty \exp\left(
  C
  - c \wb{\weakconvexfunc} \epsilon k^{1 - \steppow}
  + \log \frac{1}{\stepsize_k}\right) < \infty.
\end{align*}
The Borel-Cantelli lemma
thus implies that for any $m \in \N$
\begin{equation}
  \label{eqn:poster-session}
  \frac{1}{\stepsize_k} D_k \indic{
    \max\{D_m, \ldots, D_{k-1}\}
    \le \frac{\lambda_0}{(1 + \epsilon) \wb{\weakconvexfunc}}}
  \cas 0.
\end{equation}

For the assertion of the proposition,
for each $m \in \N$ define the
random variables
\begin{equation*}
  V_{k,m} \defeq \frac{1}{(1 - \lambda_1)^{k+1}} D_{k+1}^2
  \cdot \indic{\max_{m \le i \le k} \{D_i\}
    \le \frac{\lambda_0}{(1 + \epsilon) \wb{\weakconvexfunc}}}.
\end{equation*}
We again trace the argument for the proof
of Proposition~2 of the paper~\cite{AsiDu18};
we have by inequality~\eqref{eqn:one-step-weak-contraction} that
\begin{align*}
  \E[V_{k,m} \mid \mc{F}_{k-1}]
  & \le \frac{D_k^2}{(1 - \lambda_1)^k}
  \max\left\{1, \frac{1 - \stepsize_k / D_k}{1 - \lambda_1}\right\}
  \indic{\max_{m \le i < k} \{D_i\} \le \frac{\lambda_0}{(1 + \epsilon)
      \wb{\weakconvexfunc}}}.
\end{align*}
But the convergence~\eqref{eqn:poster-session} guarantees that
$D_k < \stepsize_k$ when $\max_{m \le i < k} \{D_i \} \le
\frac{\lambda_0}{(1 + \epsilon) \wb{\weakconvexfunc}}$, at least for
all large enough $k$, so that
$\E[V_{k,m} \mid \mc{F}_{k-1}] \le V_{k-1,m} + E_{k-1,m}$, where $E_{k,m} = 0$
eventually (with probability $1$). The
Robbins-Siegmund lemma~\ref{lemma:robbins-siegmund} completes the proof.

\bibliography{bib}

\begin{thebibliography}{37}
\providecommand{\natexlab}[1]{#1}
\providecommand{\url}[1]{\texttt{#1}}
\expandafter\ifx\csname urlstyle\endcsname\relax
  \providecommand{\doi}[1]{doi: #1}\else
  \providecommand{\doi}{doi: \begingroup \urlstyle{rm}\Url}\fi

\bibitem[Abadi et~al.(2015)Abadi, Agarwal, Barham, Brevdo, Chen, Citro,
  Corrado, Davis, Dean, Devin, Ghemawat, Goodfellow, Harp, Irving, Isard, Jia,
  Jozefowicz, Kaiser, Kudlur, Levenberg, Man\'{e}, Monga, Moore, Murray, Olah,
  Schuster, Shlens, Steiner, Sutskever, Talwar, Tucker, Vanhoucke, Vasudevan,
  Vi\'{e}gas, Vinyals, Warden, Wattenberg, Wicke, Yu, and Zheng]{TensorFlow15}
M.~Abadi, A.~Agarwal, P.~Barham, E.~Brevdo, Z.~Chen, C.~Citro, G.~S. Corrado,
  A.~Davis, J.~Dean, M.~Devin, S.~Ghemawat, I.~Goodfellow, A.~Harp, G.~Irving,
  M.~Isard, Y.~Jia, R.~Jozefowicz, L.~Kaiser, M.~Kudlur, J.~Levenberg,
  D.~Man\'{e}, R.~Monga, S.~Moore, D.~Murray, C.~Olah, M.~Schuster, J.~Shlens,
  B.~Steiner, I.~Sutskever, K.~Talwar, P.~Tucker, V.~Vanhoucke, V.~Vasudevan,
  F.~Vi\'{e}gas, O.~Vinyals, P.~Warden, M.~Wattenberg, M.~Wicke, Y.~Yu, and
  X.~Zheng.
\newblock {TensorFlow}: Large-scale machine learning on heterogeneous systems,
  2015.
\newblock URL \url{https://www.tensorflow.org/}.
\newblock Software available from tensorflow.org.

\bibitem[Asi and Duchi(2018)]{AsiDu18}
H.~Asi and J.~C. Duchi.
\newblock Stochastic (approximate) proximal point methods: Convergence,
  optimality, and adaptivity.
\newblock \emph{arXiv:1810.05633 [math.OC]}, 2018.

\bibitem[Belkin et~al.(2018)Belkin, Hsu, and Mitra]{BelkinHsMi18}
M.~Belkin, D.~Hsu, and P.~Mitra.
\newblock Overfitting or perfect fitting? {R}isk bounds for classification and
  regression rules that interpolate.
\newblock \emph{arXiv:1806.05161 [stat.ML]}, 2018.

\bibitem[Bertsekas(2011)]{Bertsekas11}
D.~P. Bertsekas.
\newblock Incremental proximal methods for large scale convex optimization.
\newblock \emph{Mathematical Programming, Series B}, 129:\penalty0 163--195,
  2011.

\bibitem[Candes and Recht(2008)]{CandesRe08}
E.~J. Candes and B.~Recht.
\newblock Exact matrix completion via convex optimization.
\newblock \emph{Foundations of Computational Mathematics}, 9:\penalty0
  717--772, 2008.

\bibitem[Clevert et~al.(2015)Clevert, Unterthiner, and
  Hochreiter]{ClevertUnHo16}
D.-A. Clevert, T.~Unterthiner, and S.~Hochreiter.
\newblock Fast and accurate deep network learning by exponential linear units
  (elus).
\newblock \emph{arXiv:1511.07289 [cs.LG]}, 2015.

\bibitem[Collins et~al.(2016)Collins, Sohl-Dickstein, and
  Sussillo]{CollinsSoSu16}
J.~Collins, J.~Sohl-Dickstein, and D.~Sussillo.
\newblock Capacity and trainability in recurrent neural networks.
\newblock \emph{arXiv:1611.09913 [stat.ML]}, 2016.

\bibitem[Davis and Drusvyatskiy(2019)]{DavisDr19}
D.~Davis and D.~Drusvyatskiy.
\newblock Stochastic model-based minimization of weakly convex functions.
\newblock \emph{SIAM Journal on Optimization}, 29\penalty0 (1):\penalty0
  207--239, 2019.

\bibitem[Davis et~al.(2018)Davis, Drusvyatskiy, MacPhee, and
  Paquette]{DavisDrMaPa18}
D.~Davis, D.~Drusvyatskiy, K.~MacPhee, and C.~Paquette.
\newblock Subgradient methods for sharp weakly convex functions.
\newblock \emph{Journal of Optimization Theory and Applications}, 179\penalty0
  (3):\penalty0 962--982, 2018.

\bibitem[Davis et~al.(2019)Davis, Drusvyatskiy, Kakade, and Lee]{DavisDrKaLe19}
D.~Davis, D.~Drusvyatskiy, S.~Kakade, and J.~D. Lee.
\newblock Stochastic subgradient method converges on tame functions.
\newblock \emph{Foundations of Computational Mathematics}, to appear, 2019.
\newblock URL \url{https://arXiv.org/abs/1804.07795}.

\bibitem[Deng et~al.(2009)Deng, Dong, Socher, Li, Li, and
  Fei-Fei]{DengDoSoLiLiFe09}
J.~Deng, W.~Dong, R.~Socher, L.~Li, K.~Li, and L.~Fei-Fei.
\newblock Image{N}et: a large-scale hierarchical image database.
\newblock In \emph{Proceedings of the IEEE Conference on Computer Vision and
  Pattern Recognition}, 2009.

\bibitem[Drusvyatskiy(2018)]{Drusvyatskiy18}
D.~Drusvyatskiy.
\newblock The proximal point method revisited.
\newblock \emph{SIAG/OPT Views and News}, to appear, 2018.
\newblock URL \url{http://www.arXiv.org/abs/1712.06038}.

\bibitem[Drusvyatskiy and Lewis(2018)]{DrusvyatskiyLe18}
D.~Drusvyatskiy and A.~Lewis.
\newblock Error bounds, quadratic growth, and linear convergence of proximal
  methods.
\newblock \emph{Mathematics of Operations Research}, 43\penalty0 (3):\penalty0
  919--948, 2018.

\bibitem[Duchi and Ruan(2018{\natexlab{a}})]{DuchiRu18a}
J.~C. Duchi and F.~Ruan.
\newblock Solving (most) of a set of quadratic equalities: Composite
  optimization for robust phase retrieval.
\newblock \emph{Information and Inference}, iay015, 2018{\natexlab{a}}.

\bibitem[Duchi and Ruan(2018{\natexlab{b}})]{DuchiRu18c}
J.~C. Duchi and F.~Ruan.
\newblock Stochastic methods for composite and weakly convex optimization
  problems.
\newblock \emph{SIAM Journal on Optimization}, 28\penalty0 (4):\penalty0
  3229--3259, 2018{\natexlab{b}}.

\bibitem[Duchi and Ruan(2019)]{DuchiRu19}
J.~C. Duchi and F.~Ruan.
\newblock Asymptotic optimality in stochastic optimization.
\newblock \emph{Annals of Statistics}, To Appear, 2019.

\bibitem[Duchi et~al.(2011)Duchi, Hazan, and Singer]{DuchiHaSi11}
J.~C. Duchi, E.~Hazan, and Y.~Singer.
\newblock Adaptive subgradient methods for online learning and stochastic
  optimization.
\newblock \emph{Journal of Machine Learning Research}, 12:\penalty0 2121--2159,
  2011.

\bibitem[Fletcher(1982)]{Fletcher82}
R.~Fletcher.
\newblock A model algorithm for composite nondifferentiable optimization
  problems.
\newblock \emph{Mathematical Programming Study}, 17:\penalty0 67--76, 1982.

\bibitem[Fletcher and Watson(1980)]{FletcherWa80}
R.~Fletcher and G.~A. Watson.
\newblock First and second order conditions for a class of nondifferentiable
  optimization problems.
\newblock \emph{Mathematical Programming}, 18:\penalty0 291--307, 1980.

\bibitem[Ghadimi and Lan(2013)]{GhadimiLa13}
S.~Ghadimi and G.~Lan.
\newblock Stochastic first- and zeroth-order methods for nonconvex stochastic
  programming.
\newblock \emph{SIAM Journal on Optimization}, 23\penalty0 (4):\penalty0
  2341--2368, 2013.

\bibitem[He et~al.(2016)He, Zhang, Ren, and Sun]{HeZhReSu16}
K.~He, X.~Zhang, S.~Ren, and J.~Sun.
\newblock Deep residual learning for image recognition.
\newblock In \emph{Proceedings of the IEEE Conference on Computer Vision and
  Pattern Recognition}, pages 770--778, 2016.

\bibitem[Hiriart-Urruty and Lemar\'echal(1993)]{HiriartUrrutyLe93ab}
J.~Hiriart-Urruty and C.~Lemar\'echal.
\newblock \emph{Convex {A}nalysis and {M}inimization {A}lgorithms {I} \& {II}}.
\newblock Springer, New York, 1993.

\bibitem[Khosla et~al.(2011)Khosla, Jayadevaprakash, Yao, and
  Li]{KhoslaJaYaLi11}
A.~Khosla, N.~Jayadevaprakash, B.~Yao, and F.-F. Li.
\newblock Novel dataset for fine-grained image categorization: Stanford dogs.
\newblock In \emph{Proc. CVPR Workshop on Fine-Grained Visual Categorization
  (FGVC)}, volume~2, page~1, 2011.

\bibitem[Krizhevsky and Hinton(2009)]{KrizhevskyHi09}
A.~Krizhevsky and G.~Hinton.
\newblock Learning multiple layers of features from tiny images.
\newblock Technical report, University of Toronto, 2009.

\bibitem[LeCun et~al.(2015)LeCun, Bengio, and Hinton]{LeCunBeHi15}
Y.~LeCun, Y.~Bengio, and G.~Hinton.
\newblock Deep learning.
\newblock \emph{Nature}, 521\penalty0 (7553):\penalty0 436--444, 2015.

\bibitem[Nemirovski and Yudin(1983)]{NemirovskiYu83}
A.~Nemirovski and D.~Yudin.
\newblock \emph{Problem Complexity and Method Efficiency in Optimization}.
\newblock Wiley, 1983.

\bibitem[Nemirovski et~al.(2009)Nemirovski, Juditsky, Lan, and
  Shapiro]{NemirovskiJuLaSh09}
A.~Nemirovski, A.~Juditsky, G.~Lan, and A.~Shapiro.
\newblock Robust stochastic approximation approach to stochastic programming.
\newblock \emph{SIAM Journal on Optimization}, 19\penalty0 (4):\penalty0
  1574--1609, 2009.

\bibitem[Patrascu and Necoara(2017)]{PatrascuNe17}
A.~Patrascu and I.~Necoara.
\newblock Nonasymptotic convergence of stochastic proximal point algorithms for
  constrained convex optimization.
\newblock \emph{arXiv:1706.06297 [math.OC]}, 2017.

\bibitem[Polyak and Juditsky(1992)]{PolyakJu92}
B.~T. Polyak and A.~B. Juditsky.
\newblock Acceleration of stochastic approximation by averaging.
\newblock \emph{SIAM Journal on Control and Optimization}, 30\penalty0
  (4):\penalty0 838--855, 1992.

\bibitem[Real et~al.(2018)Real, Aggarwal, Huang, and Le]{RealAgHuLe18}
E.~Real, A.~Aggarwal, Y.~Huang, and Q.~V. Le.
\newblock Regularized evolution for image classifier architecture search.
\newblock \emph{arXiv:1802.01548 [cs.NE]}, 2018.

\bibitem[Robbins and Monro(1951)]{RobbinsMo51}
H.~Robbins and S.~Monro.
\newblock A stochastic approximation method.
\newblock \emph{Annals of Mathematical Statistics}, 22:\penalty0 400--407,
  1951.

\bibitem[Robbins and Siegmund(1971)]{RobbinsSi71}
H.~Robbins and D.~Siegmund.
\newblock A convergence theorem for non-negative almost supermartingales and
  some applications.
\newblock In \emph{Optimizing Methods in Statistics}, pages 233--257. Academic
  Press, New York, 1971.

\bibitem[Rockafellar and Wets(1998)]{RockafellarWe98}
R.~T. Rockafellar and R.~J.~B. Wets.
\newblock \emph{Variational Analysis}.
\newblock Springer, New York, 1998.

\bibitem[Schechtman et~al.(2015)Schechtman, Eldar, Cohen, Chapman, Miao, and
  Segev]{SchechtmanElCoChMiSe15}
Y.~Schechtman, Y.~C. Eldar, O.~Cohen, H.~N. Chapman, J.~Miao, and M.~Segev.
\newblock Phase retrieval with application to optical imaging.
\newblock \emph{IEEE Signal Processing Magazine}, pages 87--109, May 2015.

\bibitem[Simonyan and Zisserman(2014)]{SimonyanZi14}
K.~Simonyan and A.~Zisserman.
\newblock Very deep convolutional networks for large-scale image recognition.
\newblock \emph{arXiv:1409.1556 [cs.CV]}, 2014.

\bibitem[Zoph and Le(2016)]{ZophLe16}
B.~Zoph and Q.~V. Le.
\newblock Neural architecture search with reinforcement learning.
\newblock \emph{arXiv:1611.01578 [cs.LG]}, 2016.

\bibitem[Zoph et~al.(2018)Zoph, Vasudevan, Shlens, and Le]{ZophVaShLe18}
B.~Zoph, V.~Vasudevan, J.~Shlens, and Q.~V. Le.
\newblock Learning transferable architectures for scalable image recognition.
\newblock In \emph{Proceedings of the IEEE conference on computer vision and
  pattern recognition}, pages 8697--8710, 2018.

\end{thebibliography}
\bibliographystyle{abbrvnat}

\end{document}